\documentclass[10pt,journal]{IEEEtran}
\usepackage{cite}
\usepackage{amsthm}
\usepackage{amsmath,amssymb,amsfonts}
\usepackage{algorithmic}
\usepackage{xcolor}
\usepackage{graphicx}
\graphicspath{ {./figs/} }
\usepackage{tabularx}
\usepackage{textcomp}
\usepackage{pgfplots}
\pgfplotsset{compat=newest}
\usepgfplotslibrary{fillbetween}
\usepackage{adjustbox} 
\usepackage{enumerate}
\usepackage{hyperref} 
\usepackage{cleveref} 
\usepackage{subcaption}
\usepackage{placeins} 
\usepackage{booktabs}    
\usepackage{subcaption}  

\tikzset{
    Ultra Thick/.style={line width=3.0pt},
    ULTRA Thick/.style={line width=4.0pt}
}

\newtheorem{thm}{Theorem}

\newtheorem{lem}{Lemma}
\newtheorem{rem}{Remark}
\newtheorem{asm}{Assumption}
\newtheorem{defi}{Definition}
\newtheorem{prop}{Proposition}
\newtheorem{prb}{Problem}
\newtheorem{t_mod}{Traffic Model}

\begin{document}

\title{Stability of vehicular admission control schemes in urban traffic networks under modelling uncertainty}

\author{Michalis Ramp$^{a}$, Andreas Kasis$^{a}$, and Stelios Timotheou$^{a}$
\thanks{This work is supported by the European Union (i. ERC, URANUS, No. 101088124, and ii. Horizon 2020 Teaming, KIOS CoE, No. 739551), and the Government of the Republic of Cyprus through the Deputy Ministry of Research, Innovation, and Digital Strategy. Views and opinions expressed are however those of the author(s) only and do not necessarily reflect those of the European Union or the European Research Council Executive Agency. Neither the European Union nor the granting authority can be held responsible for them. }
\thanks{$^{a}$ All the authors are with {KIOS Research and Innovation Center of Excellence}, {1 Panepistimiou Avenue}, {2109 Aglantzia}, {Nicosia}, {Cyprus}.
Emails:
\tt\{ramp.michalis,kasis.andreas,
timotheou.stelios\}@ucy.ac.cy
}
}

\maketitle

\begin{abstract}
Urban transportation networks face significant challenges due to traffic congestion, leading to adverse environmental and socioeconomic impacts. 
Vehicular admission control (VAC) strategies have emerged as a promising solution to alleviate congestion. 
By leveraging information and communication technologies, VAC strategies regulate vehicle entry into the network to optimize different traffic metrics of interest over space and time. 
Despite the significant development of VAC strategies, their stability at the presence of modelling uncertainty remains under-explored.
This paper investigates the stability properties of a class of decentralized VAC schemes under modelling uncertainty. 
Specifically, we consider large-scale, heterogeneous urban traffic networks characterised by nonlinear dynamics and concave macroscopic fundamental diagrams with bounded uncertainty between flow, density, and speed.
In this context, we examine a broad class of decentralized VAC dynamics, described by general nonlinear forms. 
Using passivity theory, we derive scalable, locally verifiable conditions on the design of VAC schemes, that enable stability guarantees in the presence of modelling uncertainty. 
Several examples are presented to illustrate the applicability of the proposed design framework.
Our analytical results are validated through numerical simulations on a 6 and a 20-region system, demonstrating their effectiveness and practical relevance.
\end{abstract}


\section{Introduction\label{sec:intro}}

\IEEEPARstart{T}{raffic} congestion remains one of the most critical challenges in urban transportation networks, leading to severe societal, economic, and environmental consequences.
As urbanization accelerates and vehicle ownership rises, traffic demand increasingly exceeds available infrastructure capacity, particularly during peak hours, causing congestion. This imbalance can result in excessive delays and, in extreme cases, \textit{gridlock}, where vehicular movement comes to a complete standstill\cite{causes2}. To address these issues, researchers have extensively explored congestion mitigation strategies, ranging from traffic control measures, such as traffic signal optimization and variable speed limits, to demand-side interventions that influence traveller decisions on route choice and departure timing. 
Despite significant advancements, modelling and efficiently managing traffic in multi-regional urban networks remain challenging problems due to the dynamic interactions between different network regions, the unpredictability of human behaviour, and the limitations of existing infrastructure. 

Research on traffic flow modelling has been an active field since the first studies that identified the fundamental relations in 1934 and 1935\cite{Greenshields1934THEPM,
Greenshields1935ASO}. Early developments led to the emergence of both microscopic and macroscopic traffic models in the 1950s\cite{Reuschel1950,
Chandler1958,
Leslie1961}. \textit{Microscopic models} focus on the behaviour of individual vehicles, capturing interactions such as car-following and lane-changing dynamics. In contrast, \textit{macroscopic models} describe traffic at an aggregated level, such as the road link level, treating vehicles as a continuous flow and using accumulated quantities like vehicle density, flow, and speed to represent the system’s state. Link-level macroscopic models are often preferred over microscopic models for large-scale real-time traffic control applications due to their reduced complexity\cite{Orosz2010TrafficJD,
WageningenKessels2015GenealogyOT}. However, their representational effectiveness is frequently undermined by significant uncertainties arising from complex vehicle interactions in urban environments.


A key advancement was the development of the \textit{macroscopic fundamental diagram} (MFD), which provides a low-scatter, unimodal relationship between traffic flow and vehicle accumulation in homogeneous traffic networks. The concept was first suggested by empirical observations in\cite{Godfrey1969}, which showed that every traffic network is characterized by a unique minimum average trip speed dictated by its physical structure. This laid the foundation for subsequent studies\cite{Herman1979ATA,
DAGANZO200749,
GEROLIMINIS2008759,
Helbing2009}, which confirmed that urban traffic could be effectively described using MFDs. The MFD framework enables the partitioning of large-scale urban networks into homogeneous \textit{regions}, where each region exhibits its own distinct MFD, capturing the relationship between flow and vehicle accumulation\cite{DAGANZO200749}. This regional partitioning enables the development of effective real-time traffic control strategies\cite{RAMEZANI20151,ZHANG20131,DAGANZO200749,geroliminis2013}, as MFD-based approaches are robust to demand fluctuations and provide an aggregated yet effective representation of network dynamics.

Traffic control aims to optimize vehicular flows across the network\cite{papageorgioureview}, employing strategies such as \textit{route guidance} to redistribute traffic\cite{geroleminisisik,MENELAOU2017,Lei2020}, and \textit{gating} and \textit{perimeter control} (PC) to regulate inflows in protected regions\cite{papageorgiou1990dynamic,papageorgioureview,ABOUDOLAS2013265,HADDAD2012,HADDAD2014,KOUVELAS2017,LI2020311,Lei2020,NING2023}. 
However, while protected regions remain congestion-free, other areas of the network, including the perimeter, experience exacerbated congestion\cite{kouvelas2018hierarchical}. 
To address such problems Model Predictive Control (MPC) approaches have been utilized\cite{mpc,geroliminis2013,SIRMATEL2021}, which typically combine one or more of these strategies in a hierarchical manner to improve traffic conditions\cite{Zhou2016,FU2017,Boufous2020}. 
While offering promising results, MPC schemes often overlook stability and uncertainty considerations and are computationally intensive.

Despite the plethora of traffic control strategies and modern enabling technologies, congestion continues to be a persistent challenge\cite{spectrum2019}. This is primarily driven by factors such as demand exceeding infrastructure capacity, inherent uncertainties in traffic modelling, and the occurrence of disruptive events. \textit{Vehicular admission control} (VAC) strategies have emerged as promising solutions with the potential to significantly reduce or even eliminate congestion. By harnessing information and communication technologies, VAC regulates vehicle entry into the traffic network, such as by instructing vehicles when to depart from their origin, to optimize key traffic metrics\footnote{While other works characterize \textit{vehicle departure time selection} as a demand management strategy, we consider it as a VAC strategy. The reason is to emphasize that the objective is not to modify demand itself, but rather to regulate its entry into the system.}. This ensures that waiting vehicles remain outside the network, preventing excessive accumulation, maintaining smooth traffic flow, and allowing users to fully utilize their time at their origin before departure\cite{MENELAOU2017}.
Although VAC may seem similar to PC, there is significant difference between these two control approaches.
In particular, PC regulates the flows exchanged between a protected region and its neighbours, while VAC does not. 
Instead, VAC controls the spatio-temporal distribution of inflow in different regions from the outside world, but not the inter-regional flows.

Early research on VAC focused on optimizing individual vehicle departure times and routes in a single region\cite{MENELAOU2017,menelaouTRR,menelaouTIV}. Leveraging road reservations to predict future network conditions, \cite{MENELAOU2017} considered the minimization of travel times on a per-vehicle basis, while ensuring that the regional density remained below the \textit{critical density}, where maximum outflow is observed. Expanding on this approach,\cite{menelaouTRR} introduced a composite objective to promote regional homogeneity, while\cite{menelaouTIV} enhanced the model by predicting link speeds instead of assuming free-flow conditions. To accommodate large networks with multiple regions, several studies have adopted a macroscopic VAC framework that integrates regional dynamics\cite{MENELAOU2022, MENELAOU2023, RAMP2024}. The works in\cite{MENELAOU2022, MENELAOU2023} investigated triangular and nonlinear MFDs within MPC frameworks, developing tailored optimization algorithms to solve the resulting nonconvex problems, but without considering uncertainty or stability. 

Developing VAC strategies that account for uncertainty and guarantee stability remains a significant challenge. The critical role of stability was emphasized in\cite{RAMP2024}, which demonstrated that stability may be compromised when a region is operated near the critical density point, where even minor disturbances, such as small increases in demand or flow irregularities, may lead to congestion or even gridlock. 
Hence, an ongoing challenge is the identification of suitable conditions on VAC dynamics that ensure traffic network stability.
Moreover, such conditions are important to be scalable, enabling VAC designs that require no recalibration when network dynamics or topology change, and distributed, minimizing computational and communication demands while avoiding a single point of failure.
Another significant challenge arises from modelling uncertainty in MFDs, which stems from factors such as traffic heterogeneity, poor data quality, erratic driver behaviour, and unpredictable events like accidents and changing weather conditions. 
These challenges underscore the need to investigate the design and stability properties of VAC schemes in large-scale urban transportation networks where regional MFDs are uncertain.

\noindent\textbf{Contributions.}\quad This paper investigates the behaviour of traffic networks with heterogeneous regional, nonlinear, MFD dynamics in the presence of modelling uncertainty. 
The considered system is coupled with a broad class of nonlinear VAC dynamics.  
Using passivity theory, we derive conditions on the VAC dynamics that guarantee asymptotic stability. 
These conditions are distributed, relying only on locally available information, locally verifiable, and applicable to arbitrary connected network configurations, ensuring scalability. 
The applicability of the proposed approach is demonstrated with five VAC scheme examples that satisfy the developed conditions.
Numerical simulations on a 6 and a 20-region system highlight the effectiveness and practicality of the approach by showing network stability under suitable VAC schemes, even in the presence of modelling uncertainty and user non-adherent behaviour.
	
To the authors' knowledge, this is the first study that: 
\begin{enumerate}[(i)]
\item Analytically studies the stability properties of large-scale traffic networks at the presence of modelling uncertainty.

\item Develops a framework for the design of VAC schemes, based on scalable, locally verifiable conditions, such that stability is guaranteed for large-scale traffic networks with uncertain dynamics.
\end{enumerate}

{\noindent\textbf{Paper structure.}\quad The remainder of the paper is organized as follows. \Cref{sec:problem_formulation} introduces the dynamics of the traffic network and outlines the problem statement. \Cref{sec:con_dyn} formally defines the class of VAC dynamics under consideration and characterizes the equilibria of the feedback interconnection. \Cref{sec:stab_cond} presents the main stability result, while \Cref{sec:con_ex} provides five examples of control dynamics that comply with the analysis, demonstrating the applicability and versatility of the derived stability conditions. \Cref{sec:num_res} showcases the practicality and effectiveness of the proposed VAC framework through simulations on a 6-region and a 20-region network. 
Finally, conclusions are drawn in \Cref{sec:conclusions}. 
The proofs of the stability results are provided in the appendices.

\noindent\textbf{Notation.}\quad The set of real numbers is given by $\mathbb{R}$, and the set of real non-negative numbers by $\mathbb{R}_+$.
We denote vectors by bold small letters with the $i^{th}$ component of a vector $\mathbf{y}$ given by $y_i$.
The set $\mathbb{R}^n$ contains all $n$-dimensional real vectors and the set $\mathbb{S}^{n}$ contains all the real, symmetric matrices of $n\times n$ dimension.
The $n\times1$ vector with all elements equal to 0 is given by $\mathbf{0}\in\mathbb{R}^{n}$.
Capital calligraphic letters denote sets. The operation of removing an element $i$ from a set $\mathcal{K}$ is denoted by $\mathcal{K}\setminus \{i\}$. 
For $a,b\in\mathbb{R}$ and $a\leq b$, the min-max operator is given by $[x]_a^b = \max(\min(x,b),a)$.
A function $f:\mathbb{R}^n\to\mathbb{R}$ is called positive definite if $f(0) = 0$ and $f(x) > 0$ for all $x \in \mathbb{R}^n$ elsewhere, and it is negative definite if $f(0) = 0$ and $f(x) < 0$ for every $x \neq 0$. 
The equilibrium of a system with state $x$ is denoted by $x^*$.
The Laplace transform of a system with input $\rho(t)$ and output $u(t)$ is denoted by G(s).
Units of variables are stated at first mention for clarity and conciseness.

\section{Problem Formulation\label{sec:problem_formulation}}

\subsection{Model description}
In this work, the macroscopic vehicular behaviour of a traffic network is modelled by an {$n$-region} network ($n\geq2$).
\textit{Nodes} are connected via \textit{edges} and the network structure is described by a directed graph $\mathcal{G}=\{\mathcal{N},\mathcal{E}\}$ where $\mathcal{N}=\{1,2,...,n\}$ is the set of nodes and $\mathcal{E}\subseteq\mathcal{N}\times\mathcal{N}$ is the set of edges.
An edge facilitating the flow of vehicles from region~$i$ to region~$j$, is denoted by $\epsilon_{i,j}=(i,j)\in\mathcal{E}$.
\textit{Predecessor} regions that can directly send vehicles to region~$i$ belong to the set $\mathcal{P}_{i}=\{j\in\mathcal{N}:\epsilon_{j,i}\in\mathcal{E}\}$. 
\textit{Successor} regions that directly receive vehicles from region~$i$ belong to the set $\mathcal{S}_{i}=\{l\in\mathcal{N}:\epsilon_{i,l}\in\mathcal{E}\}$.
To facilitate the reader an example network demonstrating predecessor-successor connectivity is provided in Fig.~\ref{fig:p_s}.
Vehicle inflow is introduced to the network via origin regions.
In this work, it is assumed that every vehicle that enters the network, does so to reach a destination region of the network.
Flow is absorbed by a destination region when vehicles reach their destination.
Also, all regions are both destination and origin regions.

\begin{figure}[t]
\centering
\includegraphics[width=1\columnwidth]{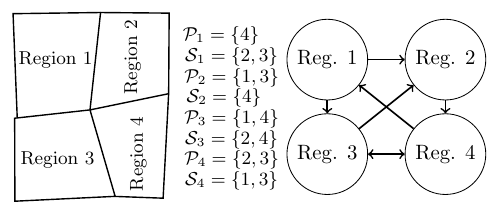}
\vspace{-2mm}
\caption{Predecessor-successor connectivity example for a vehicular network with four regions. Left-to-right: Network topology, predecessor-successor sets, node connectivity via directed edges.
}
\label{fig:p_s}
\vspace{-3mm}
\end{figure}

The regions of the network are considered to be homogeneous  in the sense that the aggregated traffic behaviour of each region~$i$\footnote{Regions are heterogeneous among them since they can have different dynamics.}, 
is captured by a nonlinear {macroscopic} fundamental diagram (MFD)\cite{DAGANZO200749,GEROLIMINIS2008759,KNOOP2015}, see Fig.~\ref{fig:tmfd},
that describes the total inter-regional flow, $g_i(\rho_i)\;[{\text{veh}}/{\text{h}}]$ of the region ($\rho_i\;[{\text{veh}}/{\text{km}}]$ denotes vehicle density) as follows
\begin{align}
g_i(\rho_i)={r_i}f_i(\rho_i)+d_i(\rho_i),\forall i\in\mathcal{N}.\label{eq:g_i}
\end{align}
The term $f_{i}:[0,\rho_i^{J}]\to\mathbb{R}_+$ is a bounded general nonlinear function of vehicle density ($\rho_i^{J}$ denotes the jam density threshold), and is equal to the product between the vehicle density of region~$i$ and vehicle speed, $\psi_i\;[{\text{km}}/{\text{h}}]$ of the same region
\begin{equation}
f_{i}(\rho_i)=\rho_i \psi_i(\rho_i),\forall i\in\mathcal{N}.
\end{equation}
The term $d_i:[0,\rho_i^{J}]\to\mathbb{R}$, captures unmodelled dynamics by describing modelling uncertainty.
Finally, $r_i\in\mathbb{R}_+$ is the region~$i$ trip completion ratio given by
\begin{equation}
r_i=L_i/ l_i,i\in\mathcal{N},
\end{equation}
with $L_i\;[\text{km}]$ being the length of region~$i$, and $l_i\;[\text{km}]$ the average trip length of a vehicle in the same region.

Next, we present the assumptions regarding the uncertainty modelling term $d_i:[0,\rho_i^{J}]\to\mathbb{R}$ and the nonlinear function $f_{i}:[0,\rho_i^{J}]\to\mathbb{R}_+$:
\begin{asm}\label{asm:d_prop}
The uncertainty term, $d_{i}(\rho_i)$:
\begin{enumerate}[(i)]
  \item is bounded and Lipschitz, i.e.,
\begin{align}
\lvert d_i(x)-d_i(y)\rvert\leq &v^{d,L}_i \lvert x-y\rvert,\forall x,y\in [0,\rho_i^J], i\in\mathcal{N},
\label{eq:Lipsc_di}
\end{align}
where $v^{d,L}_i\in\mathbb{R}_+$ is the Lipschitz constant, \label{asm_d:1}
    \item 
it takes values that always yield non-negative inter-regional flow values, i.e.,
\begin{align}
d_i(\rho_i)\geq-{r_i}f_i(\rho_i),
\forall i\in\mathcal{N}.\label{eq:pos_gi_di}
\end{align}\label{asm_d:2}
     \item \label{asm_d:3} it takes values such that $g_i(\rho_i)$, is concave with a unique maximum $g^C_i$ at $\rho_i^{C,g}$. 
\end{enumerate}
\end{asm}
\begin{asm}\label{asm:f_prop}
The nonlinear function, $f_{i}(\rho_i(t))$, is:
\begin{enumerate}[(i)]
    \item \label{asm:f_prop_0} concave in $[0,\rho_i^J]$, with a unique maximum at $\rho_i=\rho_i^{C,f}\in(0,\rho_i^J)$ where $\rho_i^{C,f}\;[{\text{veh}}/{\text{km}}]$ denotes the critical density threshold $f_i$,
    \item \label{asm:f_prop_2} Lipschitz continuous, with Lipschitz constant $v^L_i$;
i.e.,
\begin{equation}\label{eq:Lipsc}
\lvert f_i(x)-f_i(y)\rvert\leq v^L_i \lvert x-y\rvert,\forall\; x,y\in [0,\rho_i^J],i\in\mathcal{N},
\end{equation}
\end{enumerate}
\end{asm}
\Cref{asm:d_prop}\eqref{asm_d:1} ensures that $d_i:[0,\rho_i^{J}]\to\mathbb{R}$ is bounded and continuous. 
This is to ensure that the resulting inter-regional flow values will capture real observations since MFDs describe a uniform, low-scatter relationship that is not affected by the change of demand\cite{Godfrey1969,DAGANZO200749,GEROLIMINIS2008759}.
\Cref{asm:d_prop}\eqref{asm_d:2} ensures that the values of the uncertainty term will not result to negative inter-regional flow values, violating the physical behaviour that characterises MFD relationships.
\Cref{asm:d_prop}\eqref{asm_d:3} is needed to describe the behaviour of the total inter-regional flow, $g_i(\rho_i)$, namely that even in a perturbed state, the MFD still shows two distinct flow regimes (free-flow, and congested flow) with a peak that can occur at a density that can be different from $\rho_i^{C,f}$\cite{Godfrey1969,DAGANZO200749,GEROLIMINIS2008759}.
In practical settings, and despite flow fluctuations or congestion, it has been demonstrated that the peaks of $g_i(\rho_i)$ and $f_i(\rho_i)$ occur at densities that are close\cite{GEROLIMINIS2008759,KEYVANEKBATANI20121393,AMIRGHOLY2017261}.
Assumptions~\ref{asm:f_prop}(\ref{asm:f_prop_0}-\ref{asm:f_prop_2}) ensure that the nonlinear function $f_{i}(\rho_i(t))$ yields values similar to real flow measurements and result in $f_{i}(\rho_i(t))$ accurately describing the two traffic conditions that characterize the total inter-regional flow of a region~$i$, $g_i(\rho_i(t))$, see Fig.~\ref{fig:tmfd}.
For a region~$i$, these traffic conditions are defined by the density value $\rho_i^{C,g}$ that corresponds to the maximum value of $g_i(\rho_i)$, i.e., $g^C_i$.
If $\rho_i(t)\leq\rho_i^{C,g}$ the region operates at free-flow conditions, while if $\rho_i(t)>\rho_i^{C,g}$ the region operates at congestion, with the behaviour produced by $g_i(\rho_i)$ for both conditions consistent with real MFD observations\cite{GEROLIMINIS2008759}.

\begin{rem}
The assumption of Lipschitz continuity in the function $g_i$, as follows from \Cref{asm:d_prop}\eqref{asm_d:1} and \Cref{asm:f_prop}\eqref{asm:f_prop_2}, is made to support the subsequent analysis.
This is justified since, although discontinuities are present at individual edges on the microscopic level, these are effectively smoothed out at the macroscopic level, where the flows of all edges are aggregated into a continuum describing the region's overall traffic dynamics.
However, this assumption does have limitations, since it may not fully capture, non-smooth real-world events that introduce discontinuities into the MFD.
Nevertheless, at a macroscopic level, Lipschitz continuity remains a reasonable and justified assumption.
\end{rem}

Flows exiting the network through destination regions and transfer flows between regions are calculated according to the MFDs total inter-regional flow, $g_i(\rho_i)$ (see \eqref{eq:g_i}), with the inter-regional transfer flow from region~$i$ to region~$l$, given by
\begin{align}
g^{}_{il}(\rho_{i})=w_{il}g_i(\rho_{i}(t)),i\in\mathcal{N},l\in\mathcal{S}_i,\label{eq:gil}
\end{align}
with $w_{il}\in\mathbb{R}_+$ being the outflow split constants for each region that satisfy the following equality
\begin{equation}
w_{ii}+\sum_{l\in\mathcal{S}_i}w_{il}=1,\forall i\in\mathcal{N}.\label{eq:out_distr_2}
\end{equation}
Note that the outflow split constant $w_{ii}\in\mathbb{R}_+$ is the rate that vehicles end their trip in region~$i$.

Hence, the continuous-time evolution of the vehicle density state of each region~$i$, $\rho_{i}(t)$, is given by
\begin{align}\label{eq:d_rho}
\dot{\rho_i}(t)=&L_{i}^{-1}\big(-w_{ii} g_i(\rho_i(t))-\sum_{l\in\mathcal{S}_i}g_{il}(\rho_{i}(t))\nonumber\\
&+\sum_{j\in\mathcal{P}_i}g_{ji}(\rho_{j}(t))
+u_{i}(t)
\big),\forall i\in\mathcal{N}.
\end{align}
The first term in the right hand side of \eqref{eq:d_rho} is the flow exiting the network through destination region~$i$, the second term is the flow towards successor nodes, the third term is the flow from predecessor nodes and the last term, $u^{}_{i}(t)\;[{\text{veh}}/{\text{h}}]$, is the serviced demand admitted to the network through the origin region~$i$ and is considered a control variable. 

For conciseness, the traffic dynamics are defined in a compact form as follows: 
\begin{t_mod}\label{str:t_mod}
The traffic dynamics of the considered $n$-region connected traffic network are described by \eqref{eq:g_i}, \eqref{eq:gil}, \eqref{eq:d_rho} and Assumptions~\ref{asm:d_prop}, and \ref{asm:f_prop}.
\end{t_mod}
\begin{figure}[t]
\centering
\includegraphics[width=0.7\columnwidth]{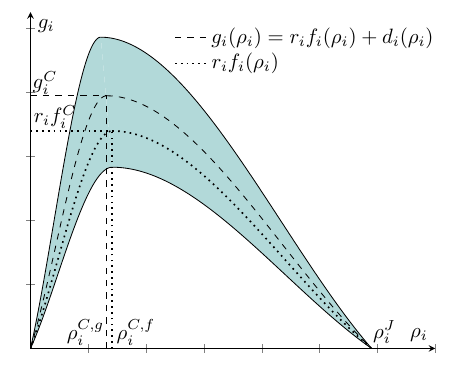}
\caption{General form of region~$i$ total inter-regional traffic flow $g_i(\rho_i)$ [veh/h], approximated by a nonlinear macroscopic fundamental diagram (MFD) area, due to flow variation and uncertainties\cite{MENELAOU2023}.  
\label{fig:tmfd}}
\end{figure}


\subsection{Problem Statement\label{sec:prob_state}}
The problem considered in this paper is the following:
\begin{prb}\label{Problem_1}
For the \Cref{str:t_mod}, develop a framework for the design of VAC schemes that:
\begin{enumerate}[(i)]
\item guarantee stability based on locally verifiable conditions,
\item are applicable on a broad range of MFD functions, following the description in \eqref{eq:g_i}, \eqref{eq:Lipsc_di},
\item rely on locally available information only,
\item are applicable to arbitrary (connected) traffic network {configurations}.
\end{enumerate}
\end{prb}

Conditions (i) and (ii) constitute the main goal of the proposed VAC approach, i.e., to ensure that the VAC scheme is supplemented by stability guarantees that are based on locally verifiable conditions characterised by robustness to modeling uncertainties (i.e. for a broad range of MFD functions).
It is also important for the VAC framework to operate in a decentralized manner and be applicable to any connected traffic network configuration.
Hence, conditions (iii), (iv) relate with scalability, i.e. the applicability of the proposed VAC approach to large-scale systems.

\section{Vehicular admission control schemes design\label{sec:con_dyn}}

The motivation for studying a class of VAC schemes follows from the high heterogeneity between regions in traffic networks.
In particular, different regions within large-scale traffic networks exhibit varying VAC capabilities due to differing infrastructure and user compliance levels.
As a result, there is a large variation in terms of VAC authority between traffic networks.
Considering a broad class of VAC dynamics enables suitable approaches for controlling both the transient behaviour and steady-state response of heterogeneous traffic networks.
Pointedly, these VAC schemes have the flexibility to be defined for fast timescales and then be altered at slower timescales, e.g. enabling different dynamic responses between seasons or days of the week.

We aim to design a class of decentralized VAC schemes that only utilize the in-region density, $\rho_i(t)$, since it is a direct measure for congestion that can be reliably estimated, as shown in\cite{WANG8283631,DARWISH2015337}.
This is done by considering the \Cref{str:t_mod} admitted vehicular demand, $u_i(t)$, as the output of the following independent VAC dynamics
\begin{subequations}\label{eq:con_dyn}
\begin{align}
\dot{x}_{\rho_i}(t)=\chi_{\rho_i}(x_{\rho_i}(t),-\rho_i(t)),\forall i\in\mathcal{N},\label{eq:con_dyn_1}\\
u_i(t)=y_{\rho_i}(x_{\rho_i}(t),-\rho_i(t)),\forall i\in\mathcal{N},
\end{align}
\end{subequations}
where $\chi_{\rho_i}:\mathbb{R}^{n_i}\times[-\rho_i^J,0]\to\mathbb{R}^{n_i}$ are locally Lipschitz and $y_{\rho_i}:\mathbb{R}^{n_i}\times[-\rho_i^J,0]\to\mathbb{R}$  are continuous for all $i\in\mathcal{N}$.
Note that \eqref{eq:con_dyn} can be region-wise heterogeneous, i.e., each region~$i$ can have different VAC dynamics.
It is assumed that for any constant input $\rho_i(t)=\bar{\rho}_i$ there exists a unique equilibrium state value $x_{\rho_i}^*$ and a corresponding output $u_i(t)=\bar{u}_i$ such that $x_{\rho_i}^*$ is a unique locally asymptotically stable equilibrium point of \eqref{eq:con_dyn}, i.e.
\begin{subequations}\label{eq:con_dyn_eq}
\begin{align}
\dot{x}_{\rho_i}(t)=0&=\chi_{\rho_i}(x^*_{\rho_i},-\bar{\rho}_i),\forall i\in\mathcal{N},\label{eq:con_dyn_eq_1}\\
u_i(t)=\bar{u}_i&=y_{\rho_i}(x_{\rho_i}^*,-\bar{\rho}_i),\forall i\in\mathcal{N},\label{eq:con_dyn_eq_2}
\end{align}
\end{subequations}
with the region of attraction of $x_{\rho_i}^*$ being denoted by $\hat{\mathcal{X}}_i$.

Additionally it is assumed that there exist static input-state characteristic maps $k^u_i:\mathbb{R}_+\to\mathbb{R},\forall i\in\mathcal{N},$ and corresponding input-output characteristic maps $k^{u,y}_i:\mathbb{R}\to\mathbb{R},\forall i$  such that
\begin{subequations}\label{eq:input_stat_map}
\begin{align}
k^u_i(-\bar{\rho}_i)=&x_{\rho_i}^*,\forall i\in\mathcal{N},\label{eq:input_stat_map_1}\\
k^{u,y}_i(-\bar{\rho}_i)=&y^u_i(k_i^u(-\bar{\rho}_i),-\bar{\rho}_i),\forall i\in\mathcal{N}.
\end{align}
\end{subequations}
We aim to characterize classes of VAC dynamics so that stability can be guaranteed for the equilibrium points of the overall feedback interconnection between \eqref{eq:con_dyn} and \Cref{str:t_mod}.

\subsection{Equilibrium points}
The equilibrium points of the overall feedback interconnection between \eqref{eq:con_dyn} and \Cref{str:t_mod} are defined next.  
\begin{defi}\label{def:eq}
The point $(\boldsymbol{x}_{\boldsymbol{\rho}}^*,\boldsymbol{\rho}^*)\in\mathbb{R}^{\sum_1^{\lvert\mathcal{N}\rvert} n_i}\times[-\rho_i^J,0]^{\lvert\mathcal{N}\rvert}$, is an equilibrium point for the overall feedback interconnection between \eqref{eq:con_dyn} and \Cref{str:t_mod} if it satisfies
{\begin{subequations}\label{eq:d_rho_eq}
\begin{align}
-w_{ii} g_i(\rho_i^*)-\sum_{l\in\mathcal{S}_i}w_{il}g_{i}(\rho_{i}^*)&\nonumber\\
+\sum_{j\in\mathcal{P}_i}w_{ji}g_{j}(\rho_{j}^*)
+u_i^*&=0,&\forall i\in\mathcal{N},\label{eq:d_rho_eq_1}\\
x_{\rho_i}^*&=k^u_i(-{\rho}^*_i),&\forall i\in\mathcal{N},\label{eq:d_rho_eq_2}\\
u_i^*&=k^{u,y}_i(-{\rho}^*_i),&\forall i\in\mathcal{N}\label{eq:u_eq}.
\end{align}
\end{subequations}}
\end{defi}
It is signified that a time frame is considered in which the VAC dynamics implicitly account for the external demand or other factors.
\begin{rem}\label{rem:int}
Onward we assume that there exists some equilibrium of the feedback interconnection between \eqref{eq:con_dyn} and \Cref{str:t_mod} satisfying \Cref{def:eq}.
Considering that the framework's target application is real-world vehicular networks, practical considerations, such that inputs are non-negative, should be taken into account on the feasibility of equilibrium points.
One way to ensure this is by suitably designing the equilibrium points $(x^*_{\rho_i},-\rho^*_i)$ and sets $\mathcal{X}_i$ and  $\mathcal{U}_i$ around these equilibria, as given in \Cref{def:passivity_def},  such that $u_i(t) \in [u_i^{\min},u_i^{\max}], t \geq 0$ where $0 \leq u_i^{\min} \leq u_i^{\max}$.
Since in the proof of \Cref{thm:1} it is showed that the trajectories of $(x_{\rho_i},-\rho_i)$ lie in a invariant subset of $\mathcal{X}_i \times  \mathcal{U}_i$, the steady-state values can be designed to satisfy $k^{u,y}_i(-\bar{\rho}_i)= y^u_i(k_i^u(-\bar{\rho}_i),-\bar{\rho}_i) \in [u_i^{\min}, u_i^{\max}]$. 
In this case, if Assumption~3 holds for the system with input $-\rho_i$ and output $u_i$, then \Cref{thm:1} follows directly.
Hence, the application of suitable VAC dynamics can ensure the applicability of the framework in a real life setting.
For example, suitable bounds on the trajectories of $u_i$ may follow by filtering the input $-\rho_i$ within \eqref{eq:con_dyn}, see example in \Cref{sec:u_adhe}.
\end{rem}

In practice, to ensure that the set-points are feasible, a network operator determines the network's operating point based on historical data or through extensive network analysis.


Density $\rho_i(t)$, is related via the static input-state characteristic maps $k_i^u(\rho_i)$ and the corresponding input-output characteristic maps $k_i^{u,y}(\rho_i)$, \eqref{eq:input_stat_map}, with the admitted vehicular demand $u_i(t)$, by the structure of the VAC dynamics \eqref{eq:con_dyn};
it affects/defines the action of the VAC dynamics \eqref{eq:con_dyn}, hence the steady-state behaviour of \Cref{str:t_mod}.
In this work a class of VAC dynamics satisfying \eqref{eq:con_dyn} will be analysed.


\subsection{Passive admission control schemes \label{sec:feed_con_asm}}

In this section, we propose suitable passivity properties, inspired from\cite{WEN1266772,KASIS7776980}, for the VAC dynamics, \eqref{eq:con_dyn}, to facilitate the analysis of their feedback interconnection with \Cref{str:t_mod}.

\begin{defi}\label{def:passivity_def}
The control dynamics, \eqref{eq:con_dyn}, are characterized as locally input strictly passive about the constant input values, $\bar{\rho}_i\in(0,\rho^J_i)$, and the corresponding steady state values, $x_{\rho_i}^*\in\mathbb{R}^{n_i}$, if there exist open neighbourhoods $\mathcal{X}_i$ of $x_{\rho_i}^*$ and $\mathcal{U}_i$ of $-\bar{\rho}_i$ and a continuously differentiable storage function $V_i(x_{\rho_i})$ such that for all $-{\rho}_i\in\mathcal{U}_i$ and $x_{\rho_i}\in\mathcal{X}_i$, the storage function $V_i(x_{\rho_i})$ satisfies the following
\begin{subequations}\label{eq:stor_fun}
\begin{align}
V_i&>0\text{ in }\mathbb{R}\backslash \{x_{\rho_i}^*\},V_i(x_{\rho_i}^*)=0,\label{eq:stor_fun_2}\\
\dot{V}_i&\leq(-{\rho}_i-(-\bar{\rho}_i))(u_i-\bar{u}_i)-\theta_i(-{\rho}_i-(-\bar{\rho}_i)),\label{eq:stor_fun_3}
\end{align}
\end{subequations}
where $\theta_i:\mathbb{R}\to\mathbb{R}_+$ is a positive definite function satisfying
\begin{align}\label{eq:phi_bound}
{\eta}_i({\rho}_i-\bar{\rho}_i)^2\leq\theta_i(-{\rho}_i-(-\bar{\rho}_i))
\end{align}
with ${\eta}_i\in\mathbb{R}_+{\setminus}\{0\}$ and $u_i$ is the output of \eqref{eq:con_dyn} with $\bar{u}_i=y_{\rho_i}(x_{\rho_i}^*,-\bar{\rho}_i)=k_i^{u,y}(-\bar{\rho}_i)$.
\end{defi}

This work considers control dynamics of the form in \eqref{eq:con_dyn} that satisfy the  local passivity conditions presented in \Cref{def:passivity_def}.
It is signified that \Cref{def:passivity_def} involves only the local VAC dynamics of each region, thus constitutes a decentralized condition.
Formally, the following assumption is made. 
\begin{asm}\label{asm:con_dyn}
For each $i\in\mathcal{N}$, the VAC dynamics, \eqref{eq:con_dyn}, are locally input strictly passive about their equilibrium values $(x_{\rho_i}^*,-{\rho}^*_i)$ according to \Cref{def:passivity_def}.
\end{asm}

\begin{rem}
The passivity properties of the control dynamics, \eqref{eq:con_dyn}, can be verified easily for general linear systems (see\cite[KYP Lemma]{KHALIL}, transfer function positive realness) and for static nonlinearities.
Explicit examples are presented in \Cref{sec:con_ex}.
Additionally, it is noted that \Cref{asm:con_dyn} allows the inclusion of a large class of control dynamics that satisfy the local passivity conditions detailed in \Cref{def:passivity_def}.
\end{rem}

\subsection{Integrator-passive admission control schemes \label{sec:feed_con_integr}}

An interesting class of VAC dynamics is described by the parallel interconnection of the dynamics in \eqref{eq:con_dyn} with an integrator. Such schemes enable to accurately track desired equilibrium points, that may be associated with the system's performance, at the presence of system uncertainty.
This class of systems is described by:
\begin{subequations}\label{eq:int}
\begin{align}
\dot{x}_{\rho_i}(t)&=\chi_{\rho_i}(x_{\rho_i}(t),-\rho_i(t)),\forall i\in\mathcal{N},\label{eq:int_1}\\
\tilde{u}_i(t)&=y_{\rho_i}(x_{\rho_i}(t),-\rho_i(t)),\forall i\in\mathcal{N},\label{eq:int_2}\\
\dot{z}_i(t)&=\frac{1}{\upsilon_i}(-\rho_i(t)-(-\tilde{\rho}_i)),i\in\mathcal{N},\label{eq:int_3}\\
u_i(t)&=\tilde{u}_i(t)+z_i(t),i\in\mathcal{N},\label{eq:int_4}
\end{align}
\end{subequations}
where $\chi_{\rho_i}:\mathbb{R}^{n_i}\times[-\rho_i^J,0]\to\mathbb{R}^{n_i}$ are locally Lipschitz and $y_{\rho_i}:\mathbb{R}^{n_i}\times[-\rho_i^J,0]\to\mathbb{R}$  are continuous for all $i\in\mathcal{N}$. 
For a constant input $\rho_i(t)=\bar{\rho}_i$ such that
\begin{align}\label{eq:int_eq_cond}
\dot{z}_i(t)&=0=\frac{1}{\upsilon_i}(-\bar{\rho}_i-(-\tilde{\rho}_i)),i\in\mathcal{N},
\end{align}
it is assumed that there exists an equilibrium state $x_{\rho_i}^*$ (the region of attraction of $x_{\rho_i}^*$ is again being denoted by $\hat{\mathcal{X}}_i$), with a corresponding output $\tilde{u}_i(t)={\tilde{u}}^*_i$ such that $(x_{\rho_i}^*,-\bar{\rho}_i,z_i^*)$ is a locally asymptotically stable equilibrium point of \eqref{eq:int}. 
Such an equilibrium corresponds to an integrator-passive admission control output ${u}^*_i$ given by
\begin{align}
{u}^*_i&={\tilde{u}}^*_i+z^*_i,i\in\mathcal{N}.\label{eq:int_5}
\end{align}
\begin{asm}\label{asm:con_dyn_int}
For each $i\in\mathcal{N}$, the VAC dynamics, \eqref{eq:int_1}-\eqref{eq:int_2}, are locally input strictly passive about their equilibrium values $(x_{\rho_i}^*,-{\rho}^*_i)$ according to \Cref{def:passivity_def}\footnote{Note that the use of \Cref{def:passivity_def} here implies replacing the notation of \eqref{eq:con_dyn} to that of \eqref{eq:int_1}-\eqref{eq:int_2}.}.
\end{asm}}

The following lemma demonstrates that, when the dynamics \eqref{eq:int_1}-\eqref{eq:int_2} are passive, then their parallel interconnection with an integrator is also passive.
\begin{lem}\label{lem:int}
For each $i\in\mathcal{N}$, consider the VAC dynamics given by \eqref{eq:int},
and consider an equilibrium such that \Cref{asm:con_dyn_int} hold. 
Then \eqref{eq:int} is locally input strictly passive about the considered equilibrium point.
\end{lem}
\begin{proof}
See Appendix~\ref{sec:Ap2}.
\end{proof}

\section{Traffic Network Stability\label{sec:stab_cond}}
This section investigates the stability properties of the feedback interconnection between \Cref{str:t_mod} and the class of VAC dynamics of the form detailed in \eqref{eq:con_dyn} or \eqref{eq:int} that satisfy the assumptions presented in \Cref{sec:feed_con_asm,sec:feed_con_integr}.

\begin{thm}\label{thm:1}
Consider the feedback interconnection between \Cref{str:t_mod}, and the VAC dynamics \eqref{eq:con_dyn}.
Moreover, consider an equilibrium of \Cref{str:t_mod}, \eqref{eq:con_dyn}, following \Cref{def:eq}, with the vehicular density equilibria satisfying $-{\rho_i^*}\in(-\rho_i^J,0)$, $\forall i\in\mathcal{N}$, and where \Cref{asm:con_dyn} holds.
If there exist $\xi_{ji} \in \mathbb{R}_+{\setminus}\{0\},(i,j)\in\mathcal{E}$, such that
\begin{align}
{\eta}_i>&  v^{d,L}_i+r_i v_i^{L}+\sum_{j\in\mathcal{P}_i}\frac{a_{ji}}{2 \xi_{ji}}+\sum_{j\in\mathcal{S}_i}\frac{\xi_{ij}a_{ij}}{2},i\in\mathcal{N},\label{eq:eta}
\end{align}
holds, where $a_{ji}$ are given by
\begin{align}
a_{ji}=&w_{ji}(r_j   v^L_j+v^{d,L}_j),(i,j)\in\mathcal{E},\label{eq:a_ji}
\end{align}
then there exists an open neighbourhood $\bar{\mathcal{M}}$ of the considered equilibrium, such that solutions to \Cref{str:t_mod}, and \eqref{eq:con_dyn}, initiated in $\bar{\mathcal{M}}$ asymptotically converge to the set of equilibrium points.
\end{thm}
\begin{proof}
See Appendix~\ref{sec:Ap1}.
\end{proof}

As already mentioned in \Cref{sec:feed_con_integr} including integrator dynamics in VAC may enable improved performance, particularly when system dynamics are uncertain.
Hence an interesting extension of \Cref{thm:1} is the interconnection of \Cref{str:t_mod} with an integrator-passive VAC scheme, \eqref{eq:int}, as given below:
\begin{prop}\label{prop:int}
Consider the feedback interconnection between \Cref{str:t_mod}, and \eqref{eq:int}. 
Moreover, consider an equilibrium of \Cref{str:t_mod}, \eqref{eq:int}, satisfying $-{\rho_i^*}\in(-\rho_i^J,0)$, and let \Cref{asm:con_dyn_int} hold.
Then, if there exists $\xi_{ji} \in \mathbb{R}_+{\setminus}\{0\},(i,j)\in\mathcal{E}$, such that \eqref{eq:eta} holds, then there exists an open neighbourhood $\bar{\mathcal{M}}$ of the considered equilibrium, such that solutions to \Cref{str:t_mod}, \eqref{eq:int}, initiated in $\bar{\mathcal{M}}$ asymptotically converge to the set of equilibrium points.

\end{prop}
\begin{proof}
\Cref{asm:con_dyn_int} allows \Cref{lem:int} to hold.
Under these conditions, the stability of the feedback interconnection between \Cref{str:t_mod}, and \eqref{eq:int} trivially follows from \Cref{thm:1}, if $(x_{\rho_i}^*,-{\rho}^*_i,z_i^*)$ satisfy \eqref{eq:d_rho_eq_1} of \Cref{def:eq}.
\end{proof}
\begin{rem}\label{rem:sp}
It should be noted that the ability to design VAC dynamics from a rich class of possible schemes enables to incorporate additional characteristics in the control design associated with the traffic network performance.
In particular,  designing the static maps of the VAC dynamics may enable to incorporate suitable optimality considerations on the system's equilibrium values, such as operating at free-flow conditions or maximizing throughput.
Moreover, suitable VAC scheme designs, adhering to the conditions of \Cref{thm:1}, may optimize the transient  behaviour of traffic networks.
\Cref{prop:int} gives further flexibility in the design of VAC schemes. 
For example, it enables effective steady state set-point regulation under unknown traffic network dynamics.
These highlight the important advantages of enabling a large class of VAC schemes and the  flexibility and broad applicability of the approach.
Another important point is that the conditions in \Cref{thm:1} rely on locally verifiable criteria.
Hence, for designing the VAC scheme for each region, only local information is needed, making the scheme applicable to large-scale networks. 
Moreover, since no assumption, other than connectivity, is made on the network topology, the presented results hold for any connected network configuration.
Although the passivity-based Lyapunov stability framework for large-scale systems is well established, this work applies the approach to transportation systems, a domain in which it has not been previously applied or demonstrated.
\end{rem}

Concluding, since the developed framework for the design of VAC schemes (i) is supplemented by stability guarantees that are based on locally verifiable conditions, (ii) is applicable to a broad range of MFD functions, (iii) employs locally available information only, and (iv) is applicable to arbitrary connected traffic network configurations, it satisfies all the requirements detailed in \Cref{Problem_1}.


\section{Examples of Vehicular admission control Dynamics\label{sec:con_ex}} 
To demonstrate the applicability of the presented analysis,  this section provides examples of VAC dynamics that satisfy the developed conditions. 
It should be noted that the presented stability results hold when different VAC schemes are implemented at different regions.
Moreover it is assumed that the controllers do not hit the actuation bounds.

\subsection{Decentralized proportional control regulation\label{sec:con_ex_1}} 
The first example is that of a decentralized proportional-control VAC scheme of the following form
\begin{align}\label{eq:con_prop}
u_i(t)=[c_i+\eta_i(-\rho_i)]_0^{u_{i}^{\max}},c_i,\eta_i\in\mathbb{R}_+{\setminus}\{0\}, i\in\mathcal{N}.
\end{align}

This controller is composed by the constant term $c_i$ and a proportional term $\eta_i(-\rho_i)$.
The intuition behind this controller is that a range of local equilibrium values can be attained by proper selection of the constants $c_i,\eta_i\in\mathbb{R}_+\setminus\{0\}$ since $c_i=u_i^*+\eta_i\rho_i^*$ and at steady state \eqref{eq:con_prop} is equivalent to
\begin{align}\label{eq:con_prop*}
u_i(t)=[u_i^*-\eta_i(\rho_i-\rho_i^*)]_0^{u_{i}^{\max}}, i\in\mathcal{N}.
\end{align}
Hence, $c_i,\eta_i\in\mathbb{R}_+\setminus\{0\}$ can be selected by taking into account optimality considerations.

It is signified that these dynamics do not have an internal state, i.e. according to the structure documented in \eqref{eq:con_dyn}, $\dot{x}_{\rho_i}(t)=0$.
It is easy to see that the control dynamics satisfy the conditions documented {in} \Cref{def:passivity_def} and \Cref{asm:con_dyn}. 
Hence, \Cref{thm:1} can be employed to deduce stability of the feedback interconnection between \Cref{str:t_mod} and \eqref{eq:con_prop}, i.e. if $\eta_i$ satisfies \eqref{eq:eta} $\forall i\in\mathcal{N}$.


\subsection{Decentralized proportional control regulation with static non-linearity \label{sec:con_ex_2}} 
 
The second decentralized VAC example to be considered has the same VAC structure with \eqref{eq:con_prop} but with an added static non-linearity term as follows
\begin{align}\label{eq:con_prop_non}
u_i(t)=[c_i+\eta_i(-\rho_i)-\phi_i(\rho_i)]_0^{u_{i}^{\max}},c_i,\eta_i\in\mathbb{R}_+{\setminus}\{0\},
\end{align}
where $\phi_i:[0,\rho_i^J]\to\mathbb{R}_+$ is a monotonically increasing, Lipschitz continuous function, $\forall i\in\mathcal{N}$. 

The first two terms of the controller allow a range of local equilibrium values to be attained while the static nonlinearity term has a softening action as  vehicle density increases.
Specifically, the monotonically increasing function causes the third term to decrease as $\rho_i$ increases (a static vanishing nonlinearity). 
This implies that the input value decreases to help protect the network from congestion as the density rises.

Using that $\phi_i$ is Lipschitz and monotonically increasing it is easy to see that the VAC dynamics satisfy \Cref{asm:con_dyn}.
Thus, via \Cref{thm:1} stability of the feedback interconnection between \Cref{str:t_mod} and \eqref{eq:con_prop_non} can be deduced if $\eta_i$ satisfies \eqref{eq:eta} $\forall i\in\mathcal{N}$.

\subsection{Decentralized first-order control dynamics\label{sec:con_ex_3}} 
The third VAC scheme to be considered is characterized by first-order dynamics with a proportional term
\begin{subequations}\label{eq:con_1st_ord}
\begin{align}
\dot{x}_{\rho_i}(t)=&-{\tau_i}^{-1}(x_{\rho_i}(t)-\gamma_i(-\rho_i(t))-c_i),i\in\mathcal{N},\label{eq:con_1st_ord_1}\\
u_i(t)=&x_{\rho_i}(t)+\eta_i(-\rho_i(t)),i\in\mathcal{N},\label{eq:con_1st_ord_2}
\end{align}
\end{subequations}
where $\tau_i,\gamma_i,c_i,\eta_i\in\mathbb{R}_+{\setminus}\{0\},\forall i\in\mathcal{N}$.

The intuition behind this controller is that a smoothing control action is achieved (hence convergence to the systems equilibrium) dictated by the transient response of the VAC dynamics \eqref{eq:con_1st_ord_1} that govern the value of the internal state, $x_{\rho_i}(t)$ as it acts in conjunction with the proportional term $\eta_i(-\rho_i(t))$.
Depending on the available means/actuators used for the practical implementation of the control action the transient response could be characterised by a delay in the change in density.
This VAC scheme can capture this behaviour and enables more efficient traffic regulation.

These control dynamics satisfy the conditions given in \Cref{def:passivity_def}, and \Cref{asm:con_dyn} with storage function $V_i(x_{\rho_i})={\tau_i}(2\gamma_i)^{-1}{(x_{\rho_i}-x_{\rho_i}^*)^2}$ 
since the derivative of the storage function yields
\begin{subequations}
\begin{align}
\dot{V}_i=&(u_i-u_i^*)(-\rho_i-(-\rho_i^*))-(u_i-u_i^*)(-\rho_i-(-\rho_i^*))\nonumber\\
&-{\gamma_i}^{-1}(x_{\rho_i}-x_{\rho_i}^*)(x_{\rho_i}-\gamma_i(-\rho_i)-c_i)\\
\leq&(u_i-u_i^*)(-\rho_i-(-\rho_i^*))-\eta_i(-\rho_i-(-\rho_i^*))^2\label{eq:diff_vi_1}.
\end{align}
\end{subequations}
The feedback interconnection stability between \Cref{str:t_mod} and \eqref{eq:con_1st_ord} can be deduced via \Cref{thm:1}.

\subsection{Second-order linear control dynamics\label{sec:con_ex_4}} 
The fourth VAC scheme to be considered is characterized by second-order linear control dynamics 
\begin{subequations}\label{eq:con_2nd_ord}
\begin{align}
\boldsymbol{x}_{\rho_i}=&\mathbf{y}_{\rho_i}+\mathbf{c}_i\label{eq:con_2nd_ord_3},i\in\mathcal{N},\\
\dot{\boldsymbol{x}}_{\rho_i}=&
\begin{bmatrix}
-\frac{1}{\tau_{i}}&0\\
\frac{1}{\kappa_i}&-\frac{1}{\kappa_i}
\end{bmatrix}\boldsymbol{x}_{\rho_i}
+\begin{bmatrix}
\frac{1}{\tau_i}\\
0
\end{bmatrix}(-\rho_i),i\in\mathcal{N},\label{eq:con_2nd_ord_1}\\
u_i=&\begin{bmatrix}
0&1
\end{bmatrix}\boldsymbol{x}_{\rho_i}+\begin{bmatrix}
\eta_i
\end{bmatrix}(-\rho_i),i\in\mathcal{N},\label{eq:con_2nd_ord_2},
\end{align}
\end{subequations}
where a change of variables takes place since this VAC scheme also utilizes constant terms in the form of the vector $\mathbf{c}_i\in\mathbb{R}^2_+$ and $\tau_i,\kappa_i,\eta_i\in\mathbb{R}_+{\setminus}\{0\},\forall i\in\mathcal{N}$.
This controller captures the fact that control decision takes time since the constants $\tau_i,\kappa_i$ represent lags in actuation and demand variation for the VAC dynamics.

These control dynamics can be expressed as a general linear system with a state-space realization parametrized by the following matrices
\begin{align}\label{eq:con_gen_ord}
\mathbf{A}=
\begin{bmatrix}
-\frac{1}{\tau_i}&0\\
\frac{1}{\kappa_i}&-\frac{1}{\kappa_i}
\end{bmatrix},
\mathbf{B}=
\begin{bmatrix}
\frac{1}{\tau_i}\\
0
\end{bmatrix},
\mathbf{C}=
\begin{bmatrix}
0&1
\end{bmatrix},
\mathbf{D}=
\begin{bmatrix}
\eta_i
\end{bmatrix}.
\end{align}
Depending on the values of the parameters,  if ($\mathbf{A}$, $\mathbf{B}$) are controllable, and assuming that \eqref{eq:eta} holds, the conditions given in \Cref{def:passivity_def}, and \Cref{asm:con_dyn} can be efficiently verified by means of a linear matrix inequality (LMI) using the KYP Lemma\cite{KHALIL}.
Specifically, if there exist matrices $\mathbf{P}\in\mathbb{R}^{3\times3},\mathbf{M}\in\mathbb{S}^{2}$ such that 
\begin{subequations}
\begin{align}
\begin{bmatrix}
\mathbf{1}&\mathbf{0}\\
\mathbf{A}&\mathbf{B}
\end{bmatrix}^T
\begin{bmatrix}
\mathbf{0}&\mathbf{M}\\
\mathbf{M}&\mathbf{0}
\end{bmatrix}
\begin{bmatrix}
\mathbf{1}&\mathbf{0}\\
\mathbf{A}&\mathbf{B}
\end{bmatrix}+\mathbf{P}\prec0,\\
\mathbf{P}=-
\begin{bmatrix}
\mathbf{0}&\mathbf{1}\\
\mathbf{C}&\mathbf{D}
\end{bmatrix}^T
\begin{bmatrix}
\mathbf{Q}&\mathbf{S}\\
\mathbf{S}^T&\mathbf{R}
\end{bmatrix}
\begin{bmatrix}
\mathbf{0}&\mathbf{1}\\
\mathbf{C}&\mathbf{D}
\end{bmatrix},
\end{align}
\end{subequations}
with $\mathbf{Q},\mathbf{R}\in\mathbb{S}^{1},\mathbf{S}\in\mathbb{R}$, then the feedback interconnection is strictly passive with a storage function equal to $V=\boldsymbol{x}_{\rho_i}^T\mathbf{M}\boldsymbol{x}_{\rho_i}$\cite{Scherer2011LinearMI}.
As an example, we used CVX\cite{cvx,gb08} to find $\mathbf{P}$, and $\mathbf{M}$ for the parameters documented in \Cref{tab:con_par} (see VAC scheme 4).
The solver yielded
\begin{equation}
\mathbf{M} =
\begin{bmatrix}
   0.403  &  0.028\\
    0.028  &  0.022
\end{bmatrix},
\mathbf{P} =
\begin{bmatrix}
    120.5804 &  -0.0741\\
   -0.0741 & 112.4012
\end{bmatrix},
\end{equation}
and like in the other examples, the stability of the overall feedback interconnection between \Cref{str:t_mod} and \eqref{eq:con_2nd_ord} can be inferred via \Cref{thm:1}.


\subsection{Integrator-passive stabilization to a desired operating point \label{sec:set_point_stabilization}}
In a realistic setting a set point usually associated with optimal steady-state behaviour like maximum vehicular throughput or vehicular traffic network operation in a congestion-free state is desirable.
However, to perform this task, the formulation of an optimization problem that aims to optimize an application specific metric (e.g., maximizing the total vehicular throughput) is needed.
Since modelling uncertainty is present in this work, the formulation of such an optimization problem is not trivial.
Alternatively the desired set point can be inferred from historical data using prior knowledge of the system steady-state behaviour and then set integral dynamics to reach this desired set point.
As already mentioned in \Cref{rem:sp}, \Cref{prop:int} allows for steady state set-point regulation if due to uncertainty VAC parameter selection is difficult.

The flexibility of the approach is showcased next by considering the addition of an integrator with the first-order VAC dynamics of \Cref{sec:con_ex_3} as follows
\begin{subequations}\label{eq:con_1st_ord_int}
\begin{align}
\dot{x}_{\rho_i}(t)=&-{\tau_i}^{-1}(x_{\rho_i}(t)-\gamma_i(-\rho_i(t))-c_i),i\in\mathcal{N},\label{eq:con_1st_ord_int_1}\\
\tilde{u}_i(t)=&x_{\rho_i}(t)+\eta_i(-\rho_i(t)),i\in\mathcal{N},\label{eq:con_1st_ord_int_2}\\
\dot{z}_i(t)=&\frac{1}{\upsilon_i}(-\rho_i(t)-(-\rho_i^*)),i\in\mathcal{N},\label{eq:con_1st_ord_int_3}\\
u_i(t)=&\tilde{u}_i(t)+z_i(t),i\in\mathcal{N}.\label{eq:con_1st_ord_int_4}
\end{align}
\end{subequations}
The storage function from \Cref{sec:con_ex_3}, is modified by appending an additional term as follows 
$
V_i(x_{\rho_i},z_i)=\frac{\tau_i}{2\gamma_i}{(x_{\rho_i}-x_{\rho_i}^*)^2}+\frac{\upsilon_i}{2}(z_i-z_i^*)^2.\label{eq:vi_1_int}
$
Its derivative yields the same result as in \eqref{eq:diff_vi_1}.
The stability of the feedback interconnection between \Cref{str:t_mod} and \eqref{eq:con_1st_ord_int} can be deduced via \Cref{prop:int}, highlighting the flexibility and broad applicability of the proposed approach.

\subsection{VAC dynamics that respect upper and lower bound input constraints \label{sec:u_adhe}}

The next scheme to be considered is motivated by \cite{kasis2017primary} and allows to ensure that $u_i(t) \in [u_i^{\min},u_i^{\max}]$.
Consider the following input filter
\begin{equation}\label{set_point_dynamics}
p^c_i(\rho_i) = 
\begin{cases} p^{c,max}_i, \;\; \rho_i < \rho^{th,1}_i \\
-\gamma_{i}(\rho_i - \rho_i^{th,1}) + p^{c,max}_i,\; \rho^{th,1}_i  \leq  \rho_i <  \rho^{th,2}_i \\
p^{c,min}_i, \;\; \rho_i \geq \rho^{th,2}_i
\end{cases}
\end{equation}
where $\gamma^c_i \geq 0$ and $p^{c,min}_i = p^{c,max}_i -\gamma^c_{i}(\rho_i^{th,2} - \rho_i^{th,1})$.
Moreover, consider the input dynamics described below,
\begin{subequations}\label{sys_u}
\begin{align}
    u_i &= u_{1,i} + u_{2,i} + c_{i} \\
    u_{1,i}(s) &= G_i(s) p^c_i(s) \\
    u_{2,i} &= -\beta^c_i \rho_i, 
\end{align}
\end{subequations}
where $\beta^c_i > 0$ and $c_{i}$ is a design constant. 
Moreover, $G_i(s)$ describes the dynamics of the system with input $p^c_i$ and output $u_{1,i}$ in the Laplace domain, as
\begin{equation}\label{sys_G}
G_i(s) = K_i \frac{(1 + sT_{1,i})}{(1 + sT_{2,i})(1 + sT_{3,i})}
\end{equation}
where $K_i, T_{1,i} , T_{2,i}$, and $T_{3,i}$ are positive  constants.
For this system, consider that $\rho_i \in (\tilde{\rho}_i, \tilde{\tilde{\rho}}_i)$, e.g., $\tilde{\rho}_i = 0$ and $\tilde{\tilde{\rho}}_i = \rho^J_i$. 
These dynamics can be designed such that $u_i(t) \in [u_i^{\min},u_i^{\max}], t \geq 0$. 
In particular, it can be shown that the trajectories of $u_i(t)$ can be bounded by  evaluating the $\mathcal{L}_1$ norm of $G_i$ and noting that both $\rho_i$ and $p^c_i$ lie in bounded sets. 
Then, given the description of $u_{1,i}$ and $u_{2,i}$, trajectory bounds, given feasible initial conditions, may be obtained through suitable design of the constants $p^{c,min}_i, p^{c,max}_i, \gamma^c_i,$ and $c_{i}$.
Conditions that enable the passivity of the above system are presented in the following lemma, proven in the appendix.
\begin{lem}\label{lemma_passivity}
Consider the input dynamics described by \eqref{sys_u}, where $u_{1,i}$ follows from \eqref{set_point_dynamics}, \eqref{sys_G}.
Then the system with input $-\rho_i$ and output $u_i$ is input strictly passive about any equilibrium, if \eqref{set_point_dynamics} satisfies $\gamma^c_{i} < \beta^c_i/K_i $.
\end{lem}
\begin{proof}
See Appendix~\ref{sec:Ap3}.
\end{proof}
Next, the VAC dynamics examples from \Cref{sec:con_ex} will be employed on a representative simulation of a 6 and a 20-region vehicular traffic network to demonstrate the advantages of the proposed approach.

\section{Simulation Results\label{sec:num_res}}

This section presents numerical simulations, conducted on a 6-region network and a 20-region network,  under varying disturbance and uncertainty conditions. 
The main result of this work, \Cref{thm:1}, is demonstrated through the 6-region network, while the scalability and robustness of the proposed approach under uncertainty are  demonstrated through simulations on the 20-region network. 
The section concludes with a brief discussion on simulations that explore the sensitivity of the analytic stability result to controller parameter variations.

First, to illustrate the effectiveness of this work, two representative simulations using a 6-region traffic network are conducted aiming to illustrate the validity of \Cref{thm:1}, i.e. the stable network behaviour under the action of four different VAC schemes satisfying \Cref{def:passivity_def} and \Cref{asm:con_dyn}, in the presence of modeling uncertainty. 

\begin{figure}[t]
\centering
\begin{subfigure}{1\columnwidth}
\centering
\begin{tikzpicture}
\draw (0, 0) node[inner sep=0] {\includegraphics[trim={4.5cm 11.5cm 4.5cm 11.5cm},clip,width=0.55\columnwidth]{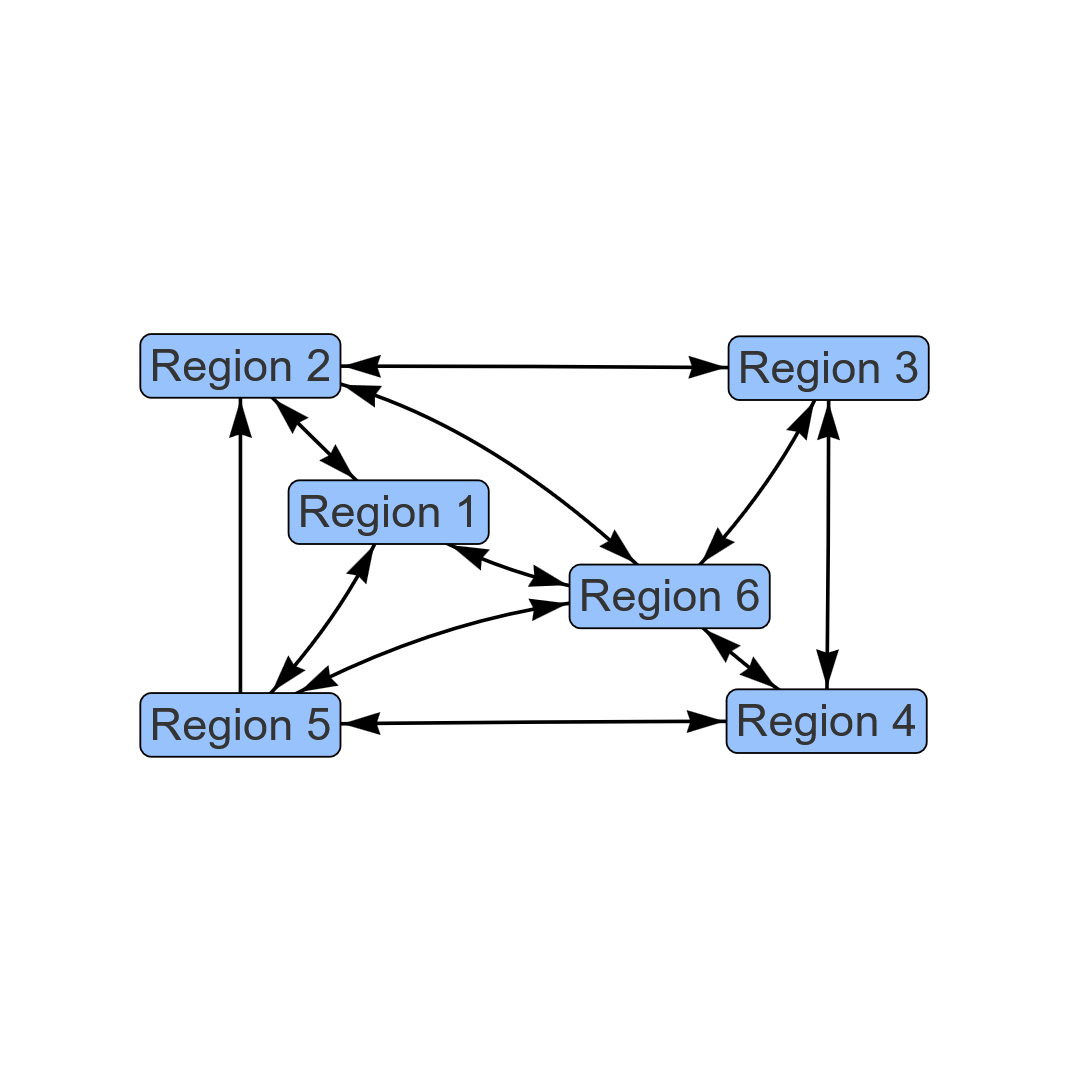}};
\draw (3.8, 1.20) node[rotate=0] {\small $\mathcal{P}_{1}=\mathcal{S}_{1}=\{2,5,6\}$};
\draw (3.57, 0.85) node[rotate=0] {\small $\mathcal{P}_{2}=\{1,3,5,6\}$};
\draw (3.44, 0.5) node[rotate=0] {\small $\mathcal{S}_{2}=\{1,3,6\}$};
\draw (3.82, 0.15) node[rotate=0] {\small $\mathcal{P}_{3}=\mathcal{S}_{3}=\{2,4,6\}$};
\draw (3.82, -0.2) node[rotate=0] {\small $\mathcal{P}_{4}=\mathcal{S}_{4}=\{3,5,6\}$};
\draw (3.44, -0.55) node[rotate=0] {\small $\mathcal{P}_{5}=\{1,4,6\}$};
\draw (3.61, -0.9) node[rotate=0] {\small $\mathcal{S}_{5}=\{1,2,4,6\}$};
\draw (4.15, -1.25) node[rotate=0] {\small $\mathcal{P}_{6}=\mathcal{S}_{6}=\{1,2,3,4,5\}$};
\end{tikzpicture}
\caption{Network predecessor-successor connectivity. Left-to-right: Node connectivity via directed edges, predecessor-successor sets.}
\label{fig:mfd-6a}
\end{subfigure}\\
\vspace{-2mm}%
\begin{subfigure}{1\columnwidth}
\centering
\includegraphics[trim={0.5cm 8.6cm 1.7cm 8cm},clip,width=0.9\columnwidth]{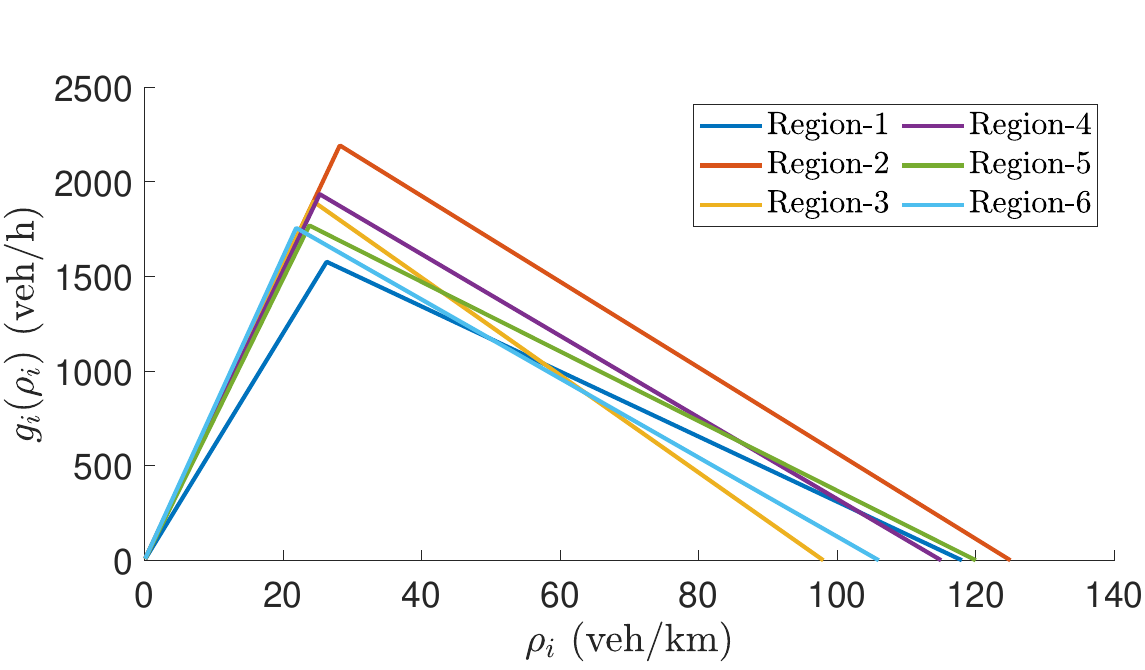}
\caption{Network MFDs produced via the parameters of \Cref{tab:params}.}
\label{fig:mfd-6b}
\end{subfigure}
\vspace{-2mm}
\caption{6-region network topology and MFD characteristics.}
\label{fig:mfd-6}
\vspace{-3mm}
\end{figure}
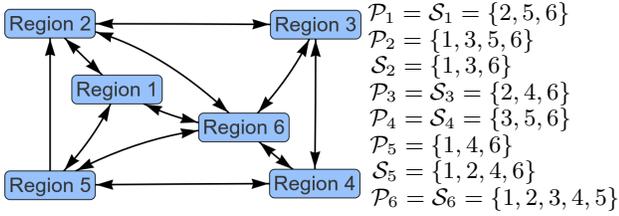
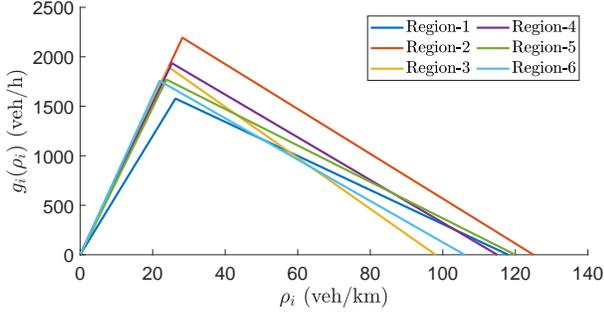

Two cases are tested employing the VAC schemes presented in \Cref{sec:con_ex}.
For Case-1, the VAC scheme described by \eqref{eq:con_prop} is employed in all the regions.
For Case-2, the VAC scheme described by \eqref{eq:con_prop} is employed in regions~1 and 2, the scheme described by \eqref{eq:con_prop_non} is employed in regions~3 and 4, the scheme described by \eqref{eq:con_1st_ord} is employed in region~5 and the scheme described by \eqref{eq:con_2nd_ord} is employed in region~6.
This aims to show that different VAC schemes can be employed in each region, without compromising the stability of the traffic network.
An integrator (see \Cref{sec:set_point_stabilization}) is included when the VAC scheme \eqref{eq:con_prop} is employed.
The goal during the simulations is to drive the system states to the values documented in \Cref{tab:con_set};
The selection of these set-points was based on prior knowledge and a performance criterion commanding maximum-throughput steady-state operation in free-flow conditions.
\begin{table}[t]
    \centering
    \caption{Feedback interconnection steady state set-points corresponding to the parameters of the VAC dynamics of \Cref{tab:con_par} }
    \label{tab:con_set}
    \begin{tabularx}{\columnwidth}{cX}
    \hline
        Simulation: & Set-point vectors:  \\
    \hline
 1,2 &       $\boldsymbol{\rho}^*{=}[17.4,
   22.9,
   24.4,
   18,
   12.5,
   21.9]^T$  \\ \hline
    \end{tabularx}
\end{table}
The parameter values for the employed VAC schemes are documented in \Cref{tab:con_par}.
\begin{table}[t]
    \centering
    \caption{Gain values }
    \label{tab:con_par}
    \begin{tabularx}{\columnwidth}{cX}
    \hline
        VAC scheme: & Gains: \\
    \hline 
        \eqref{eq:con_prop} &$\boldsymbol{\eta}{=}[63.3,
   65.1,
   83.9,
   91.5,
   73.3,
  111.4]^T$,
   $\mathbf{c}{=}[1280.5,
    2658.1,
    2677.1,
    1732.7,
    1004.0,
    2507.6]^T$, 
$\boldsymbol{\upsilon} {=} [1,1,0.001,0.001,0.002,0.001]^T$ \\
        \eqref{eq:con_prop_non} &$\eta_3=83.9,c_3=2678.5,\eta_4=91.5,c_4=1733.3,\phi_i(\rho_i)=0.001\rho_i^3, i=3,4$ \\
        \eqref{eq:con_1st_ord} &$\eta_5=48.9,\gamma_5=24.4,c_5=1004,\tau_5=0.1$ \\
        \eqref{eq:con_2nd_ord} & $\eta_6=110.4,\tau_6=0.02,\kappa_6=0.003$ \\\hline
    \end{tabularx}
\end{table}
It is noted that the gain values $\boldsymbol{\eta}$ were selected to satisfy condition \eqref{eq:eta} of \Cref{thm:1}.
The traffic network parameters are given in \Cref{tab:params}.

\begin{table}[t]
    \centering
    \caption{6-region network parameters\cite{MENELAOU2023}}
    \label{tab:params}
    \begin{tabularx}{\columnwidth}{XX}
    \hline
        Parameter \& Value \\
    \hline
    $f_i(\rho_i)=\min\{\psi^\text{f}_{i}\rho_i,\frac{\psi^\text{f}_{i}\rho_i^\text{C}}{\rho_i^\text{J}-\rho_i^\text{C}} (\rho_i^\text{J}-\rho_i)\}$ \\
     $\mathbf{L}{=}[1.2, 1, 0.85, 0.9, 1.02, 0.88]^T$  \\
     $\mathbf{l}{=}[0.6, 0.45, 0.35, 0.4, 0.48, 0.34]^T$ \\
    $\boldsymbol{\psi}^f{=}[30, 35, 32, 34, 35, 31]^T$ \\
    $\boldsymbol{\rho}^J{=}[118, 125, 98, 115, 120, 106]^T$ \\
    $\boldsymbol{\rho}^{C,f}{=}[26.3, 28.2, 24.4, 25.3, 23.8, 21.9]^T$ \\
   $\mathbf{w}{=}\begin{bmatrix}
   0.25 & 0.25 & 0 &    0 &    0.25 & 0.25\\
   0.15 & 0.35 & 0.3 &  0 &    0 &    0.2\\
   0 & 0.1 &  0.3 &  0.4 &  0 &    0.2\\
   0 & 0 &    0.24 & 0.16 & 0.3 &  0.3\\
   0.05 & 0.1 &    0 &    0.25 & 0.3 &  0.3\\
   0.32 & 0.03 & 0.23 & 0.17 & 0.1 &  0.15
    \end{bmatrix}$ \\
    $\xi_{ij}=\xi_{ji}=1$, $\forall i,j$\\
    $v^{L}_i=\max\{\psi^\text{f}_{i},\frac{\psi^\text{f}_{i}\rho_i^\text{C}}{\rho_i^\text{J}-\rho_i^\text{C}}\}$, $\forall i$\\
    $v^{d,L}_i=0.2 v^{L}_i$, $\forall i$\\
    \hline
    \end{tabularx}
\end{table}

Additionally two different forms of disturbances are employed during the simulations (applied separately). 
The first corresponds to traffic-flow signalling devices malfunction, while the second corresponds to constant driver non-adherence to the VAC commands (simulated by inflow to the network described by random noise with zero mean and standard deviation equal to 20\% of $u_i(t)$ at each control loop).
The duration of the first simulation scenario is $120\,\text{[min]}$ while for the second simulation scenario it is $60\,\text{[min]}$.


\subsection{Traffic-flow signalling devices malfunction\label{sec:num_sce_1}}
The VAC schemes ability to operate in the presence of disturbances in the form of traffic-flow signalling devices malfunction is showcased next.
At $t=30\,\text{[min]}$ and for $1.5\,\text{[min]}$, each VAC scheme is disengaged in all regions and furthermore unregulated vehicles enter the network from regions-1, 3 and 5 at capacity flow rates, thus driving some regions to congestion.
This event simulates commonly observed driver non-adherence behaviour during periods where traffic-flow signalling devices malfunction or stop working altogether and constitutes an instance of serious controller disruption/disturbance.
Pointedly, at the unregulated $1.5\,\text{[min]}$ interval the admitted demand profile $\mathbf{u}$ is given by
\begin{equation*}
\mathbf{u}(t){=}[938.9,0  ,929.2,  0,991.3,0]^T,30\leq t\leq31.5\,\text{[min]}.
\end{equation*}
When the interval passes, the VAC framework is re-engaged in order to recover the system at the desired set-points.

The results of the simulation are depicted in Figs.~\ref{fig:u}-\ref{fig:g}. 
The network state converges to the desired operating point in less than 10~[min], with only a minor overshoot observed in regions~3-6.
During the control disruption interval the system traffic density diverges from the desired operation point since vehicle inflow is unregulated.
After the controllers are re-engaged, they drive the density at every region to steady state as demonstrated in Fig.~\ref{fig:r}, obtaining the desired set-points.
Despite the control disruption (see Fig.~\ref{fig:u}, t=30~[min]) and the unregulated vehicles entering the network, and also the fact that some regions get congested 
(see Fig.~\ref{fig:v}), the employed VAC framework is able to recover the system to the desired density set-points (see Fig.~\ref{fig:r}).

More specifically, regions~3, and 6 get congested;
an increase in vehicle density and a drop in vehicle speed is observed in Figs.~\ref{fig:r}, \ref{fig:v}.
At the same time a density drop is observed due to the control disruption, see Fig.~\ref{fig:u} (t=30~[min]).
Consequently, the inter-regional outflow response is heavily affected by the disruption in control regulation, with outflow variations in almost all regions as depicted in Fig.~\ref{fig:g}.
It is signified however that a swift recovery to free-flow conditions is demonstrated, see Fig.~\ref{fig:g} (t$>$40~[min]), and gridlock is avoided;
these results demonstrate the effectiveness of the approach.

\begin{figure}[t]
\centering
\begin{subfigure}{1\columnwidth}
\centering
\includegraphics[trim={1.0cm 8.5cm 1.7cm 8cm},clip,width=1\columnwidth]{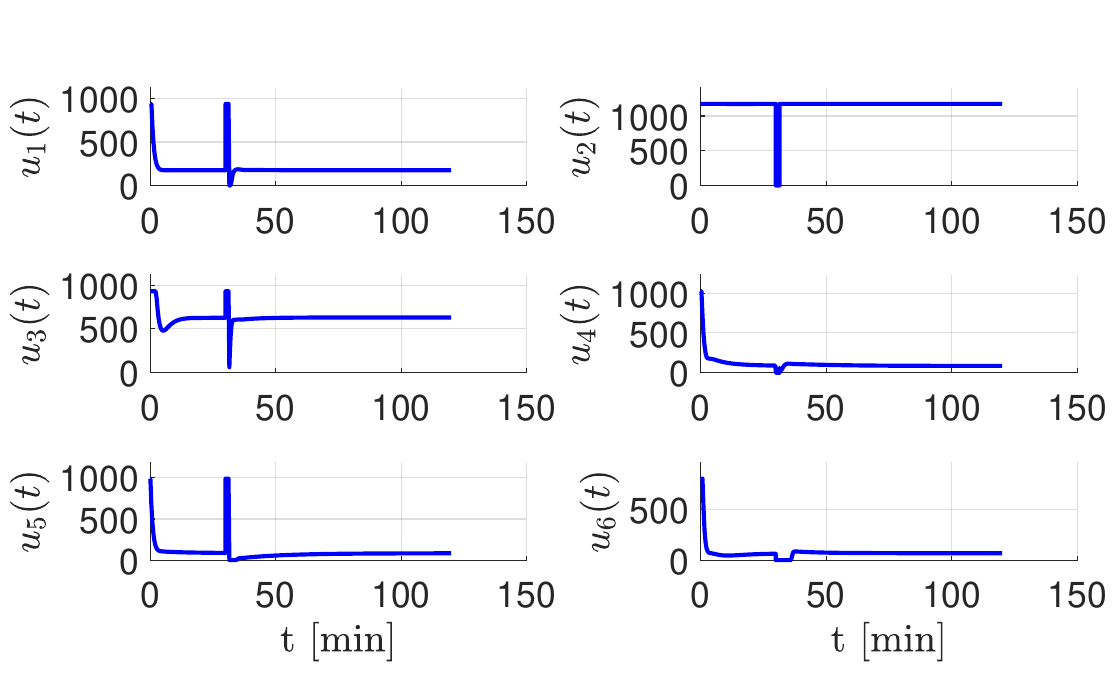}
\caption{Case-1 employing a single VAC scheme for network regulation.}
\label{fig:ua}
\end{subfigure}\\
\vspace{-2mm}%
\begin{subfigure}{1\columnwidth}
\centering
\includegraphics[trim={1.0cm 8.5cm 1.7cm 8cm},clip,width=1\columnwidth]{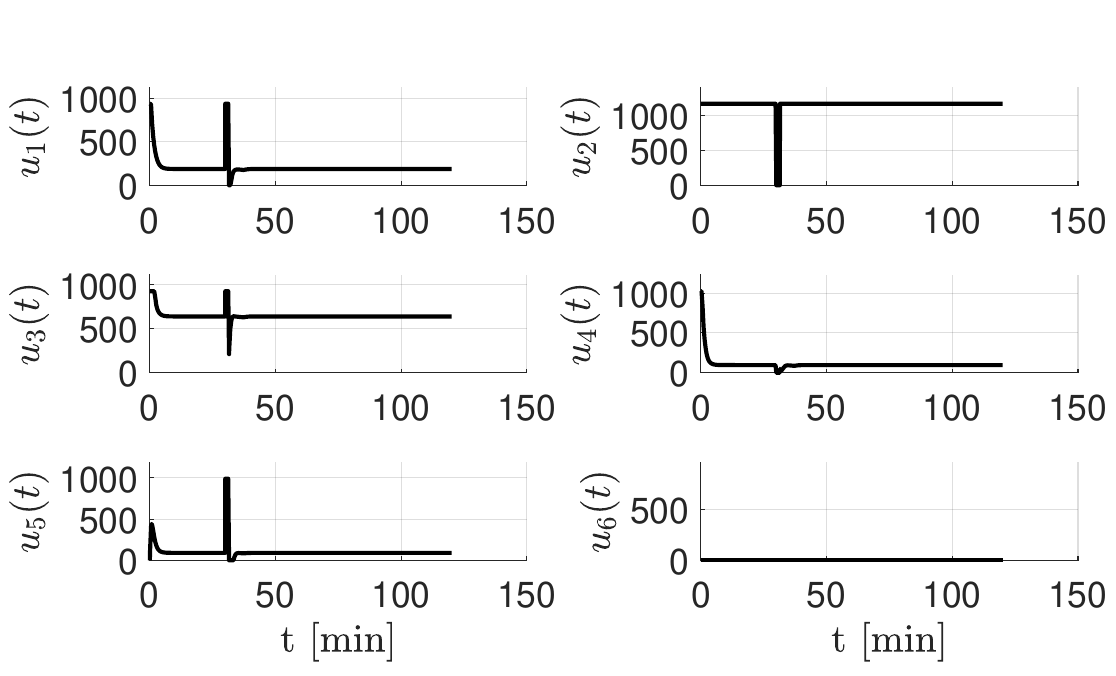}
\caption{Case-2 simultaneously employing 4 different VAC schemes for network regulation.}
\label{fig:ub}
\end{subfigure}
\vspace{-2mm}
\caption{Admitted demand $[{\text{veh}}/{\text{h}}]$. Case-1: blue line. Case-2: black dotted line.}
\label{fig:u}
\vspace{-3mm}
\end{figure}

\begin{figure}[t]
         \centering
         \includegraphics[trim={1.0cm 8.5cm 1.7cm 8cm},clip,width=0.9\columnwidth]{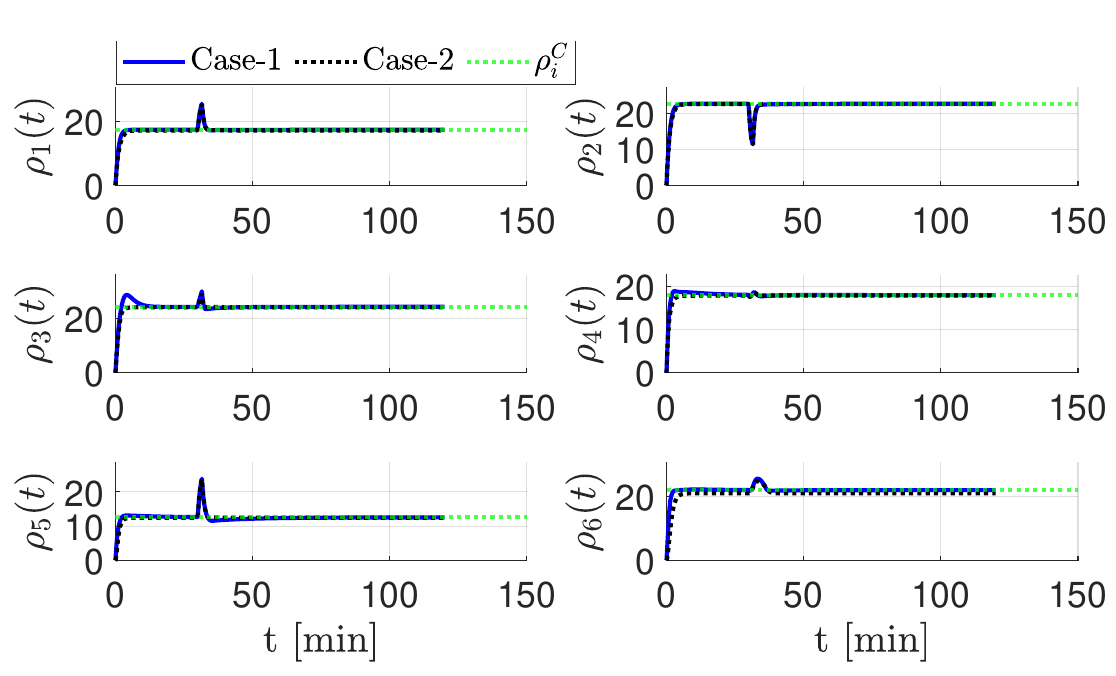}
                \vspace{-2mm}
         \caption{Network density $[{\text{veh}}/{\text{km}}]$. The VAC schemes drive the solutions at the reference set-points even from the congested region and despite control authority disruptions. Case-1: blue line. Case-2: black dotted line. }
         \label{fig:r}
                  \vspace{-3mm}
\end{figure}

\begin{figure}[t]
\centering
\includegraphics[trim={1.0cm 8.5cm 1.7cm 8cm},clip,width=0.9\columnwidth]{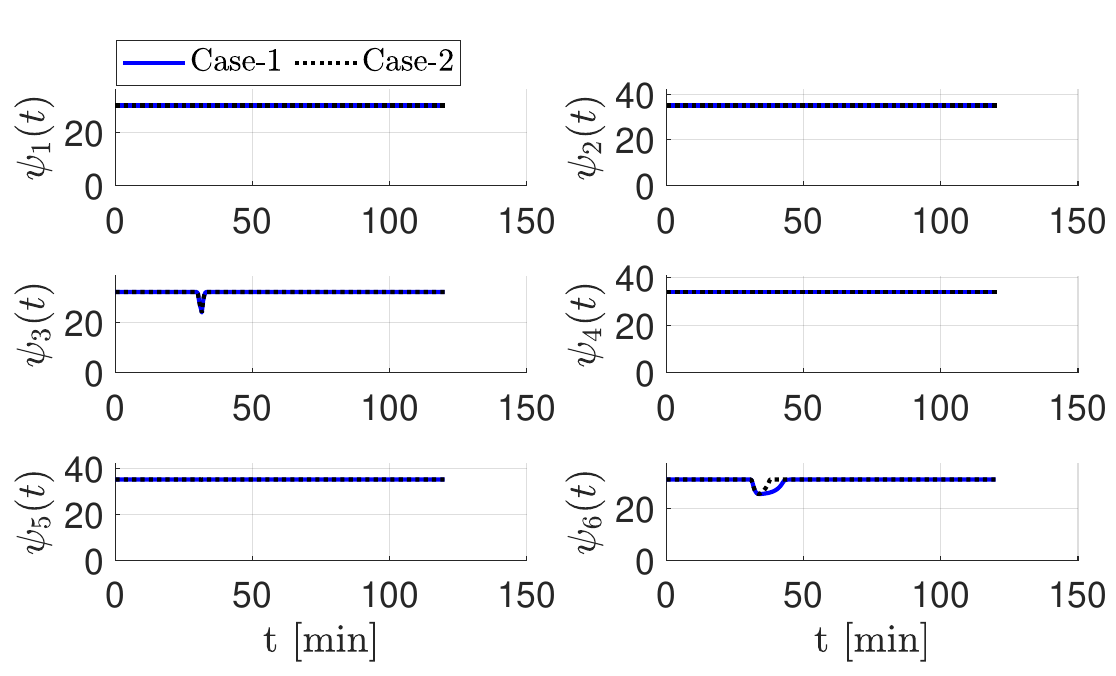}
\vspace{-2mm}
\caption{Network speed $[{\text{km}}/{\text{h}}]$. Initially, vehicles travel at free-flow speed. 
Due to control authority disruptions (see $t=30\,\text{[min]}$), for a short interval, the flow speed for regions~3, and 6 reduces below the free-flow nominal values meaning that those regions become congested. Case-1: blue line. Case-2: black dotted line.
}
\label{fig:v}
\vspace{-3mm}
\end{figure}

\begin{figure}[t]
         \centering
         \includegraphics[trim={1.0cm 8.5cm 1.7cm 8cm},clip,width=0.9\columnwidth]{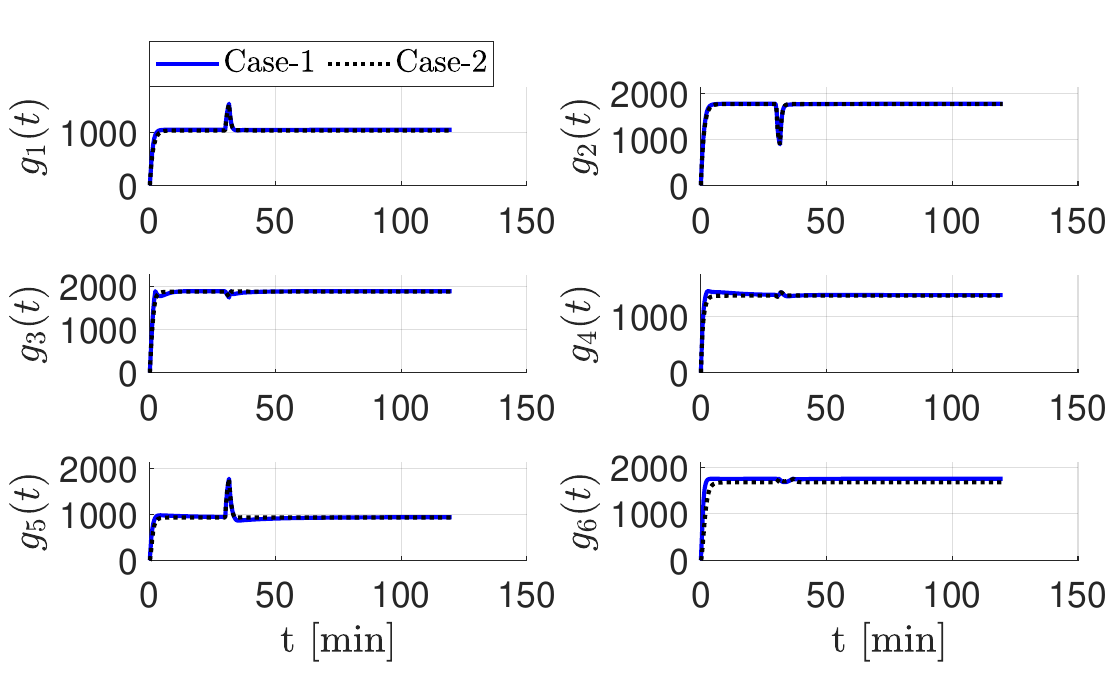}
                \vspace{-2mm}
         \caption{Inter-regional outflow $[{\text{veh}}/{\text{h}}]$. 
         Despite control authority disruptions, the outflow assumes steady-state due to the action of the VAC schemes driving the states at the desired set-points even from congestion (see $t=30\,\text{[min]}$). Case-1: blue line. Case-2: black dotted line. }
         \label{fig:g}
\end{figure}


\subsection{Continuous driver non-adherence to VAC commands}

To further test the framework's ability to operate in the presence of disturbances, the simulations are repeated, this time implementing constant driver non-adherence to VAC commands instead of traffic-flow signalling device malfunctions.
Constant driver non-adherence to VAC commands is simulated by the introduction of a disturbance inflow to the network described by uniformly distributed random noise with zero mean, $\mu=0$, and standard deviation equal to 20$\%$ of the control effort $u_i(t)$, i.e. $\sigma=0.2u_i(t),\forall i\in\mathcal{N}$.
This, simulates commonly observed delays or early departure responses of drivers to VAC commands.

The results of the simulation can be seen in Figs.~\ref{fig:ud}-\ref{fig:gd}. 
Despite constant driver non-adherence to VAC commands, the VAC framework drives the density at each region to steady state approximately in 15~[min], as demonstrated in Fig.~\ref{fig:rd}, obtaining the desired set-points.
After convergence, all the regions operate in free-flow (see Fig.~\ref{fig:vd}, t=15~[min]), and both {employed} VAC strategies efficiently regulate the system to the desired density set-points (see Fig.~\ref{fig:rd}, t$>$15~[min]).

The inter-regional outflow response has a noisy profile due to the VAC strategies reacting to  constant driver non-adherence, see Fig.~\ref{fig:gd}.
Despite this, a swift convergence to the desired operation points is observed and steady state robust response is demonstrated, see Fig.~\ref{fig:gd} (t$>$15~[min]). 

\begin{figure}[t]
         \centering
         \includegraphics[trim={1.0cm 8.5cm 1.7cm 8cm},clip,width=0.9\columnwidth]{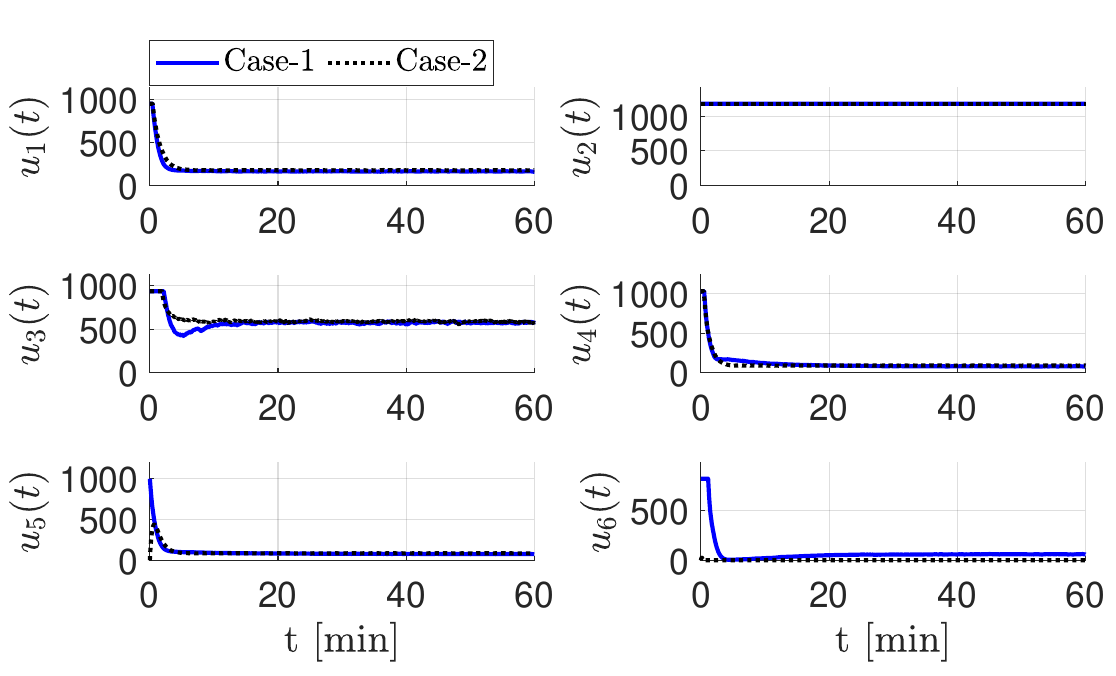}
                \vspace{-2mm}
         \caption{Admitted demand $[{\text{veh}}/{\text{h}}]$. Case-1: blue line. Case-2: black dotted line. 
         }
         \label{fig:ud}
                  \vspace{-3mm}
\end{figure}

\begin{figure}[t]
         \centering
         \includegraphics[trim={1.0cm 8.5cm 1.7cm 8cm},clip,width=0.9\columnwidth]{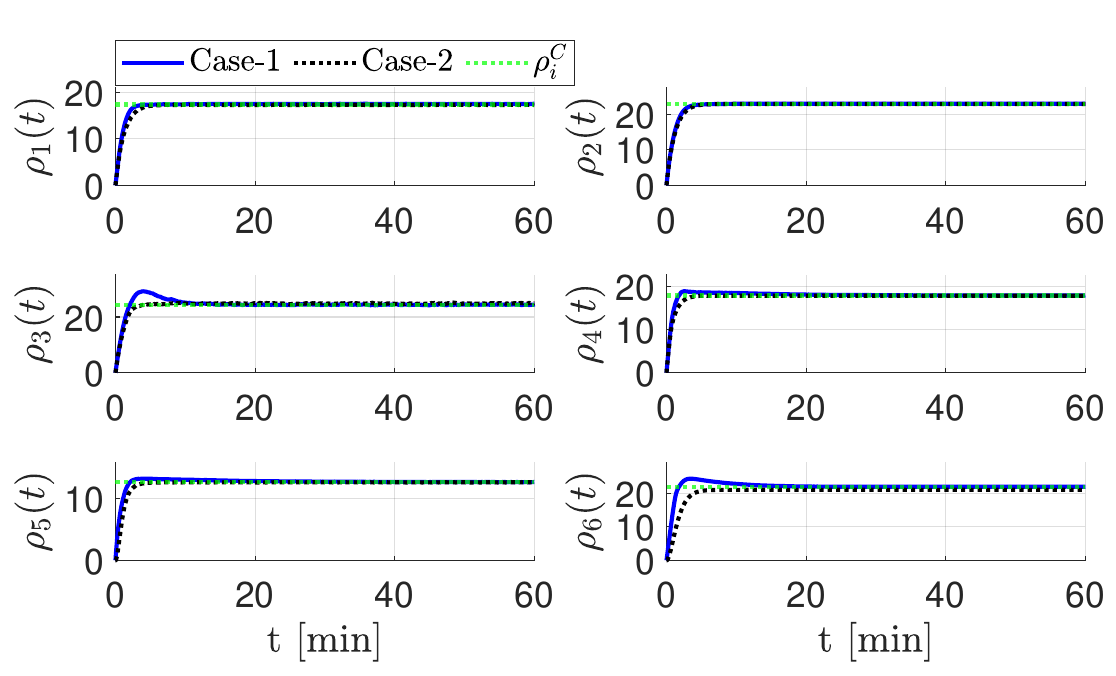}
                \vspace{-2mm}
         \caption{Network density $[{\text{veh}}/{\text{km}}]$. The VAC schemes quickly achieve the desired operation set-points despite constant user non-adherence. Case-1: blue line. Case-2: black dotted line. }
         \label{fig:rd}
                  \vspace{-3mm}
\end{figure}

\begin{figure}[t]
         \centering
         \includegraphics[trim={1.0cm 8.5cm 1.7cm 8cm},clip,width=0.9\columnwidth]{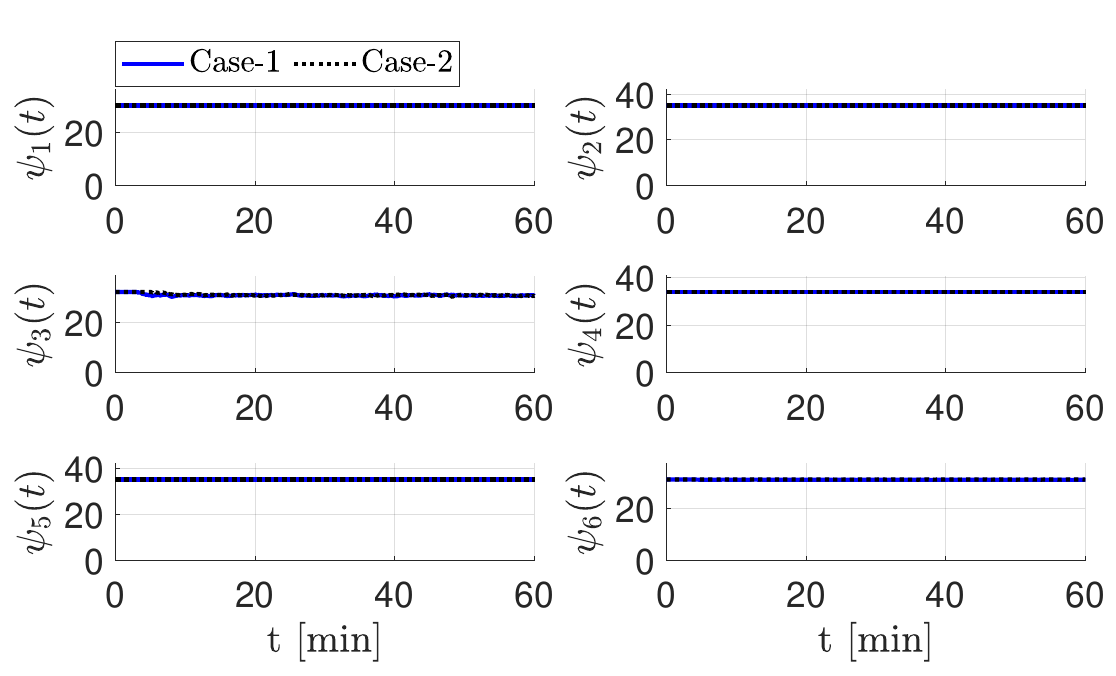}
                \vspace{-2mm}
         \caption{Network speed $[{\text{km}}/{\text{h}}]$.  Vehicles travel at free-flow speed hence congestion doesn't manifest the network. Case-1: blue line. Case-2: black dotted line. 
         }
         \label{fig:vd}
                  \vspace{-3mm}
\end{figure}

\begin{figure}[t]
         \centering
         \includegraphics[trim={1.0cm 8.5cm 1.7cm 8cm},clip,width=0.9\columnwidth]{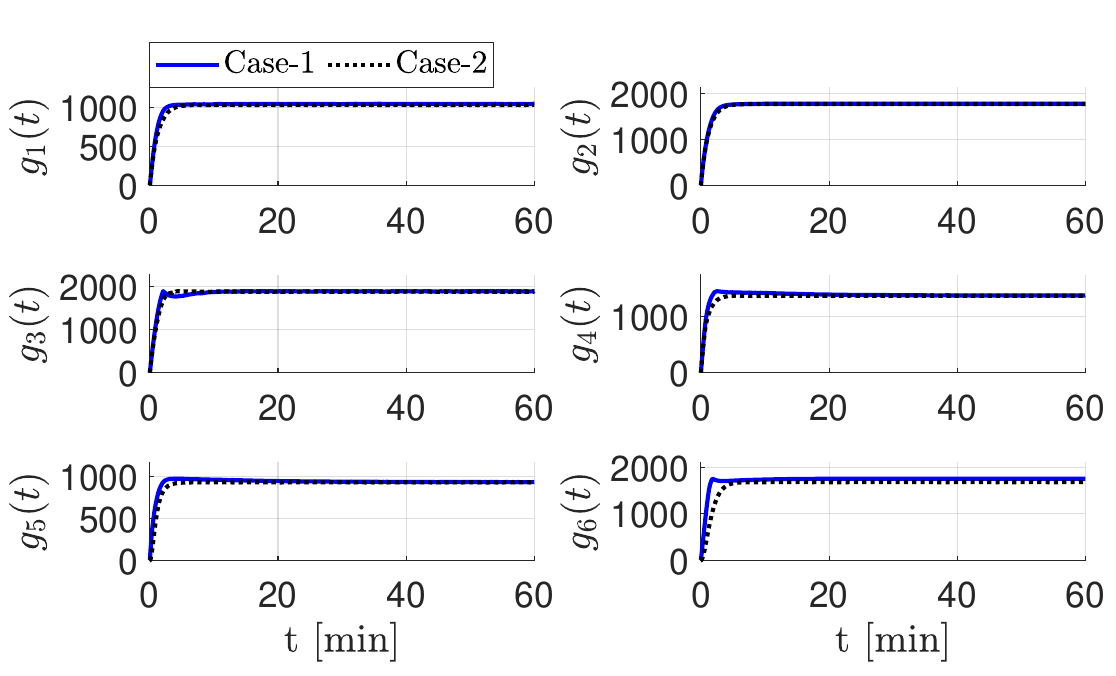}
                \vspace{-2mm}
         \caption{Inter-regional outflow $[{\text{veh}}/{\text{h}}]$. 
         The inter-regional outflow quickly assumes steady-state due to the action of the VAC schemes driving the states at the desired set-points despite constant user non-adherence. Case-1: blue line. Case-2: black dotted line. }
         \label{fig:gd}
\end{figure}


\subsection{Demonstrating Scalability\label{sec:scal}}
The following numerical simulation aims to demonstrate the scalability and large scale applicability of the developed approach.
Specifically, to demonstrate that the presented results are also applicable to larger networks, the simulation from \Cref{sec:num_sce_1} was repeated on the 20-region network shown in Fig.~\ref{fig:20n}.
In this simulation, \eqref{eq:con_prop}, \eqref{eq:con_prop_non}, \eqref{eq:con_1st_ord}, \eqref{eq:con_2nd_ord}, and \eqref{sys_u} are employed for regions~1 to 4, 5 to 8, 9 to 12, 13 to 16 and 17 to 20, respectively.

\begin{figure}[t]
\centering
\begin{subfigure}{1\columnwidth}
\centering
\includegraphics[trim={7.3cm 10.4cm 6.4cm 10cm},clip,width=0.5\columnwidth]{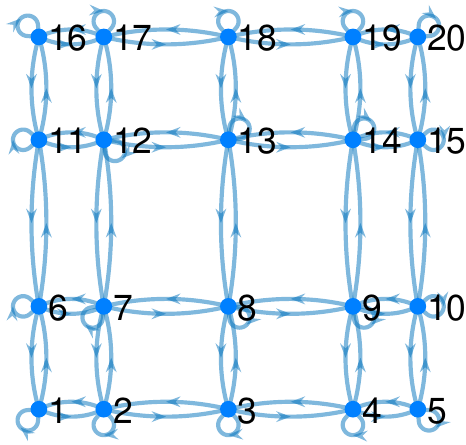}
\caption{Node connectivity via directed edges for the 20-region network employed in \Cref{sec:scal}.}\label{fig:20n}
\end{subfigure}\\
\vspace{-2mm}%
\begin{subfigure}{1\columnwidth}
\centering
\includegraphics[trim={0.0cm 8.5cm 1.7cm 8cm},clip,width=0.9\columnwidth]{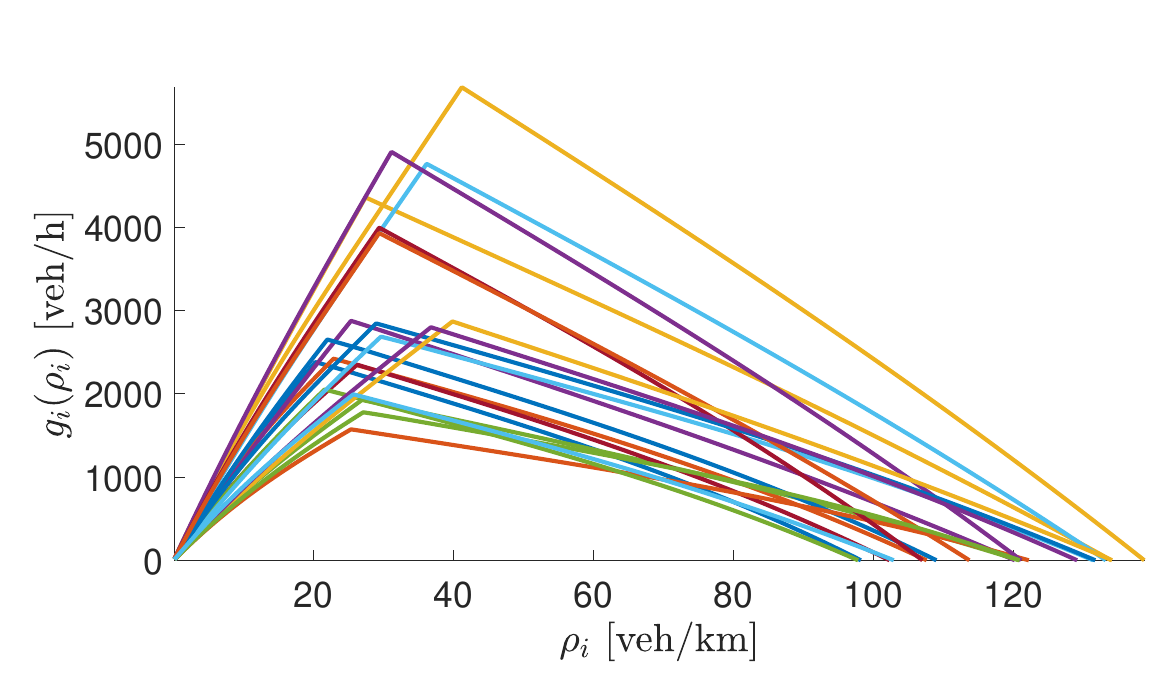}
\caption{Network MFDs produced via the parameters of \Cref{tab:params_20}.}
\label{fig:mfd-20b}
\end{subfigure}
\vspace{-2mm}
\caption{20-region network topology and MFD characteristics.}
\label{fig:mfd-20}
\vspace{-3mm}
\end{figure}

The goal of the simulation is to drive the system states to the values documented in \Cref{tab:con_set_20};
The selection of these set-points was based on prior knowledge and a performance criterion commanding maximum-throughput steady-state operation in free-flow conditions.
The traffic network parameters are given in \Cref{tab:params_20}.
The parameter values for the employed VAC schemes are documented in \Cref{tab:con_par_20}.
It is noted that the gain values $\boldsymbol{\eta}$ were selected to satisfy condition \eqref{eq:eta} of \Cref{thm:1}.
For compactness and to improve readability, the worst density deviation from $\boldsymbol{\rho}^*$ across all 20 regions is used for presentation and is given in Fig.~\ref{fig:r_max_v_7}. Figure~\ref{fig:r_max_v_7} also includes the vehicle speed for Region~7, presented as a representative plot to illustrate how congestion can occur, as well as the framework's ability to recover from it.
To explicitly provide the density-error response, and the velocity response for all the regions of the network, the logarithmic density-error heatmap transient response is given in Fig.~\ref{fig:e_r_20_heat} while heatmaps of the vehicle speeds for all regions are given in Fig.~\ref{fig:v_20_heat}.
The results are given next:
\begin{table*}[t]
\centering
\caption{Feedback interconnection steady state set-points corresponding to the parameters of the VAC dynamics of \Cref{tab:con_par_20} }
\label{tab:con_set_20}
\centering
\begin{tabularx}{\textwidth}{cX}
\hline
Region densities: & Set-point vectors:  \\
\hline
$\boldsymbol{\rho}^*$ &$[20.2,14.4,6.3,8.7,9.6,15.9,26.2,11.9,25.3,11.3,9.9,21.5,16.2,9.9,28.9,6.3,19,20.1,27.1,25.7]^T$  \\ \hline
\end{tabularx}
\end{table*}
\begin{table*}[t]
\centering
\caption{Feedback interconnection steady state set-points corresponding to the parameters of the VAC dynamics of \Cref{tab:con_par_20_uncert} }
\label{tab:con_set_20_uncert}
\centering
\begin{tabularx}{\textwidth}{cX}
\hline
Region densities: & Set-point vectors:  \\
\hline
$\boldsymbol{\rho}^*$ &$[16.08,11.36,4.96,6.8,7.36,12.8,20.96,9.44,20.24,9.04,8.24,17.28,12.8,7.92,23.12,5.12,15.36,15.84,21.68,20.56]^T$  \\ \hline
\end{tabularx}
\end{table*}
\begin{table*}[t]
    \centering
    \caption{20-region network parameters\cite{MENELAOU2023}}
    \label{tab:params_20}
    \begin{tabularx}{\textwidth}{XX}
    \hline
        Parameter \& Value \\
    \hline
    $f_i(\rho_i)=\min\{v^\text{f}_{i}\rho_i,\frac{v^\text{f}_{i}\rho_i^\text{C}}{\rho_i^\text{J}-\rho_i^\text{C}} (\rho_i^\text{J}-\rho_i)\}$ \\
     $\mathbf{L}{=}[1.11,0.95,0.92,1.16,0.98,1.19,1.07,1.06,0.86,1.18,0.99,1.04,0.98,0.91,1.16,1.04,0.82,1.05,0.94,1.02]^T$  \\
     $\mathbf{l}{=}[0.49,0.39,0.27,0.44,0.48,0.43,0.58,0.46,0.57,0.3,0.3,0.53,0.39,0.31,0.57,0.37,0.51,0.5,0.56,0.47]^T$ \\
    $\mathbf{v}^f{=}[43.29,35.36,43.63,39.35,30.24,43.94,40.36,44.96,32.59,32.06,43.99,40.45,30.99,41.33,41.31,43.85,40.67,31.86,30.3,30.39]^T$ \\
    $\boldsymbol{\rho}^J{=}[98.2,107.6,134,120.2,120.8,133.2,102.3,109,122.2,138.7,121.1,97.8,131.4,107,131.7,113.7,134.1,129.1,120.9,102.9]^T$ \\
    $\boldsymbol{\rho}^{C,f}{=}[20.23,22.83,27.4,25.33,26.89,36.14,26.19,21.94,25.32,41.16,31.1,21.55,29.6,29.36,28.91,29.35,39.83,36.75,27.08,25.66]^T$ \\
\setlength{\arraycolsep}{2pt}
\renewcommand{\arraystretch}{0.8}
$\scriptstyle\mathbf{w}{=}\scriptstyle\begin{bmatrix}
    2.543&2.462&0&0&0&4.996&0&0&0&0&0&0&0&0&0&0&0&0&0&0\\
    1.500&1.917&1.229&0&0&0&5.353&0&0&0&0&0&0&0&0&0&0&0&0&0\\
   0&1.731&1.625&1.571&0&0&0&5.073&0&0&0&0&0&0&0&0&0&0&0&0\\
   0&0&2.492&2.138&1.368&0&0&0&4.002&0&0&0&0&0&0&0&0&0&0&0\\
   0&0&0&3.213&1.365&0&0&0&0&5.422&0&0&0&0&0&0&0&0&0&0\\
   3.695&0&0&0&0&0.989&1.688&0&0&0&3.627&0&0&0&0&0&0&0&0&0\\
   0&2.543&0&0&0&2.160&0.980&0.861&0&0&0&3.455&0&0&0&0&0&0&0&0\\
   0&0&3.502&0&0&0&2.548&0.902&0.985&0&0&0&2.063&0&0&0&0&0&0&0\\
   0&0&0&3.691&0&0&0&2.357&0.677&0.523&0&0&0&2.752&0&0&0&0&0&0\\
   0&0&0&0&1.925&0&0&0&2.984&0.889&0&0&0&0&4.202&0&0&0&0&0\\
   0&0&0&0&0&2.553&0&0&0&0&2.900&1.288&0&0&0&3.259&0&0&0&0\\
   0&0&0&0&0&0&3.230&0&0&0&1.550&1.679&1.270&0&0&0&2.271&0&0&0\\
   0&0&0&0&0&0&0&1.546&0&0&0&2.854&2.225&1.138&0&0&0&2.238&0&0\\
   0&0&0&0&0&0&0&0&2.201&0&0&0&2.709&0.895&1.099&0&0&0&3.096&0\\
   0&0&0&0&0&0&0&0&0&3.771&0&0&0&1.391&1.074&0&0&0&0&3.763\\
   0&0&0&0&0&0&0&0&0&0&1.825&0&0&0&0&1.266&6.909&0&0&0\\
   0&0&0&0&0&0&0&0&0&0&0&3.197&0&0&0&1.380&1.173&4.249&0&0\\
   0&0&0&0&0&0&0&0&0&0&0&0&2.729&0&0&0&1.188&1.786&4.298&0\\
   0&0&0&0&0&0&0&0&0&0&0&0&0&1.867&0&0&0&2.418&1.857&3.858\\
   0&0&0&0&0&0&0&0&0&0&0&0&0&0&1.756&0&0&0&2.090&6.154
    \end{bmatrix}\cdot10^{-1}$\\
    $\xi_{ij}=\xi_{ji}=1$, $\forall i,j$\\
    $v^{L}_i=\max\{v^\text{f}_{i},\frac{v^\text{f}_{i}\rho_i^\text{C}}{\rho_i^\text{J}-\rho_i^\text{C}}\}$, $\forall i$\\
    $\boldsymbol{v}^{d,L}_{\min}=[25.4,23.2,18.6,20.8,20.7,18.8,24.4,22.9,20.4,18,20.6,25.6,19,23.4,19,22,18.6,19.4,20.7,24.3]^T$\\
    $\boldsymbol{v}^{d,L}_{\max}=[76.4,69.7,56,62.4,62.1,56.3,73.3,68.8,61.4,54.1,61.9,76.7,57.1,70.1,56.9,66,55.9,58.1,62,72.9]^T$\\
    \hline
    \end{tabularx}
\end{table*}

\begin{table}[t]
\centering
\caption{Gain values employed in \Cref{sec:scal}}
\label{tab:con_par_20}
\begin{tabularx}{\columnwidth}{cX}
\hline
VAC scheme: & Gains: \\
\hline 
\eqref{eq:con_prop} &$\boldsymbol{\eta}_{1,..,4}=[544.6,454.9,488,428.1]^T$\\
	& $\mathbf{c}_{1,..,4}=[11775,6411,3086,3660]^T$\\
\eqref{eq:con_prop_non} &$\boldsymbol{\eta}_{5,..,8}=[330.8,533.7,613.7,638.7]^T$\\
	&$\mathbf{c}_{5,..,8}=[3153,8530,15888,7488]^T$\\
	&$\phi_i(\rho_i)=0.001\rho_i^3, i=5,..,8$\\
\eqref{eq:con_1st_ord} &$\boldsymbol{\eta}_{9,..,12}=[234.3,426.2,268.1,224.4]^T$\\
	&$\boldsymbol{\gamma}_{9,..,12}=[117.2,213.1,134.1,112.2]$\\
	&$\mathbf{c}_{9,..,12}=[8794.8,7211.8,4432,7213.1]^T$\\
	&$\tau_i=0.1, i=9,..,12$\\
\eqref{eq:con_2nd_ord} &$\boldsymbol{\eta}_{13,..,16}=[293.3,485,270.5,612.2]^T$\\
	&$\mathbf{c}_{13,..,16}=[4702,4834.1,9062.6,3972]^T$\\
	&$\tau_i=0.1, \kappa_i=0.2, i=13,..,16$\\
\eqref{sys_u} &$\boldsymbol{\rho}^{th,1}_{17,..,20}=[0.4,0.4,0.3,0.3]^T$\\
	&$\boldsymbol{\rho}^{th,2}_{17,..,20}=[131.4,126.5,118.5,100.8]^T$\\
	&$\mathbf{p}^{c,max}_{17,..,20}=[134.1,129.1,120.9,102.9]^T$\\
	&$\boldsymbol{\beta}^c_{17,..,20}=[485.3,234.6,270.6,285.2]^T$\\
	&$\mathbf{c}_{17,..,20}=[9280,4716.1,7259.5,7419.1]^T$\\
	&$T_{1,i}=0.01,T_{2,i}=0.1,T_{3,i}=0.2, \kappa_i=0.1,\gamma_i=1.023, i=17,..,20$\\
	& $\boldsymbol{\upsilon}_{17,..,20}=[0.001,0.001,0.001,0.001]^T$\\\hline
\end{tabularx}
\end{table}

\begin{table}[t]
\centering
\caption{Gain values employed in \Cref{sec:uncert}}
\label{tab:con_par_20_uncert}
\begin{tabularx}{\columnwidth}{cX}
\hline
VAC scheme: & Gains: \\
\hline 
\eqref{eq:con_prop} &$\boldsymbol{\eta}_{1,..,4}=[356.8,532.1,457.3,532.6]^T$\\
	& $\mathbf{c}_{1,..,4}=[6591.9,6111.0,2367.2,3697.7]^T$\\
\eqref{eq:con_prop_non} &$\boldsymbol{\eta}_{5,..,8}=[253.6,541.4,378.8,472.7]^T$\\
	&$\mathbf{c}_{5,..,8}=[2005.4,7059.5,8035.0,4544.2]^T$\\
	&$\phi_i(\rho_i)=0.001\rho_i^3, i=5,..,8$\\
\eqref{eq:con_1st_ord} &$\boldsymbol{\eta}_{9,..,12}=[231.2,225.5,257.6,349.3]^T$\\
	&$\boldsymbol{\gamma}_{9,..,12}=[115.6,112.8,128.8,174.7]$\\
	&$\mathbf{c}_{9,..,12}=[7084.4,3166.8,3488.1,9127.9]^T$\\
	&$\tau_i=0.1, i=9,..,12$\\
\eqref{eq:con_2nd_ord} &$\boldsymbol{\eta}_{13,..,16}=[365,581.1,442.2,434.6]^T$\\
	&$\mathbf{c}_{13,..,16}=[4756,4708,11367,2333]^T$\\
	&$\tau_i=0.1, \kappa_i=0.2, i=13,..,16$\\
\eqref{sys_u} &$\boldsymbol{\rho}^{th,1}_{17,..,20}=[0.4,0.4,0.3,0.3]^T$\\
	&$\boldsymbol{\rho}^{th,2}_{17,..,20}=[131.4,126.5,118.5,100.8]^T$\\
	&$\mathbf{p}^{c,max}_{17,..,20}=[134.1,129.1,120.9,102.9]^T$\\
	&$\boldsymbol{\beta}^c_{17,..,20}=[397,450.8,448.2,317.9]^T$\\
	&$\mathbf{c}_{17,..,20}=[6217.5,7272.2,9773.8,6727.6]^T$\\
	&$T_{1,i}=0.01,T_{2,i}=0.1,T_{3,i}=0.2, \kappa_i=0.1,\gamma_i=1.023, i=17,..,20$\\
	& $\boldsymbol{\upsilon}_{17,..,20}=[0.001,0.001,0.001,0.001]^T$\\\hline
\end{tabularx}
\end{table}


Similar to \Cref{sec:num_sce_1}, network states (density, inter-regional flow, velocity) converge to the target operating point in under 10~[min], as shown by the worst density-error transient response (blue line) in Fig.~\ref{fig:r_max_v_7}.
Pointedly, the results show rapid density-error convergence to zero implying the system reaches the desired set point (see Fig.~\ref{fig:e_r_20_heat}). 
Vehicle speeds in all regions stabilize at nominal free-flow values (see~Fig. \ref{fig:v_20_heat});
indicatively, region~7 speed is shown in Fig.~\ref{fig:r_max_v_7} (red line).
At t=30~[min], a control disruption (unregulated vehicle inflow) causes the system traffic density to diverge, leading to congestion in regions 7, 9, 12 and 15 (see Fig.~\ref{fig:r_max_v_7} (red line) and Fig.~\ref{fig:v_20_heat})
However, as shown in the aforementioned figures, re-engagement of the controllers restores the network to the desired steady-state density points across all regions at approximately $t=$40~[min], restoring the nominal free-flow speed operation in the entire network, see Figs.~\ref{fig:r_max_v_7} and \ref{fig:v_20_heat}.
The seamless application of the VAC framework on a 20-region network and the swift convergence of the states and their fast recovery demonstrate the scalability and effectiveness of the developed framework.

\begin{figure}[t]
         \centering
         \includegraphics[trim={0.4cm 8.5cm 0.3cm 9cm},clip,width=0.8\columnwidth]{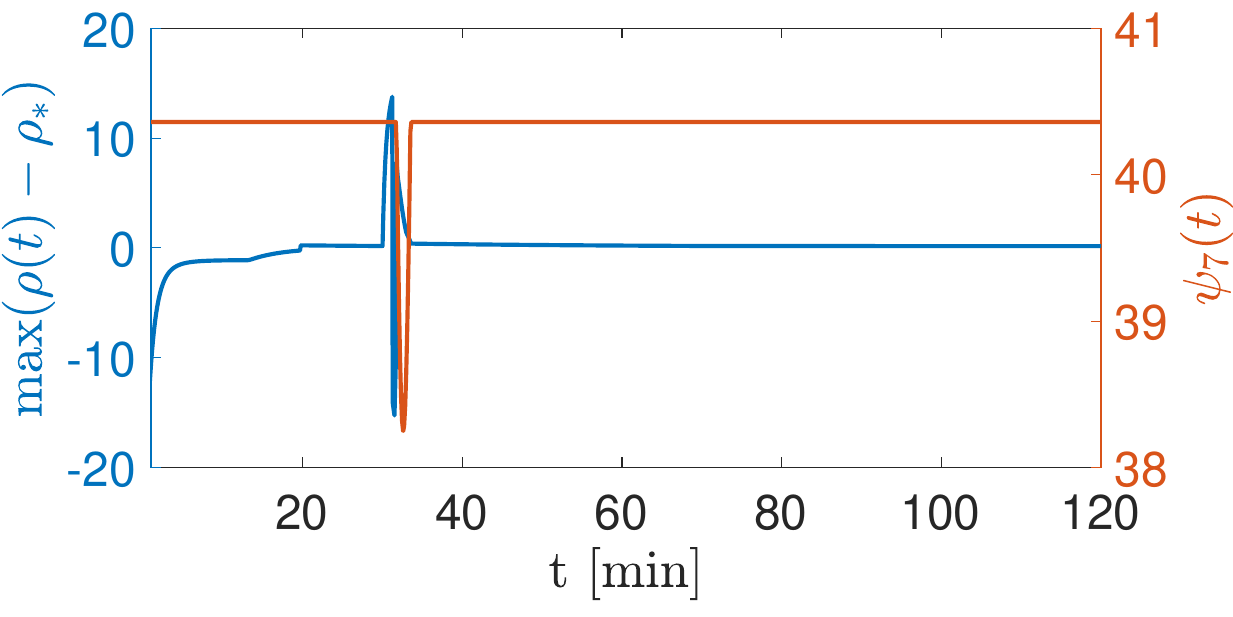}
                \vspace{-2mm}
         \caption{Blue line: Worst density deviation from $\boldsymbol{\rho}^*$ across 20 regions, $[{\text{veh}}/{\text{km}}]$. 
         Red line: Region~7 speed indicating that the region gets momentarily congested at t = 30~[min], $[{\text{km}}/{\text{h}}]$.
         }
         \label{fig:r_max_v_7}
                  \vspace{-3mm}
\end{figure}

\begin{figure}[t]
\centering
\begin{subfigure}{1\columnwidth}
\centering
\includegraphics[trim={0cm 8.5cm 1.2cm 9cm},clip,width=0.8\columnwidth]{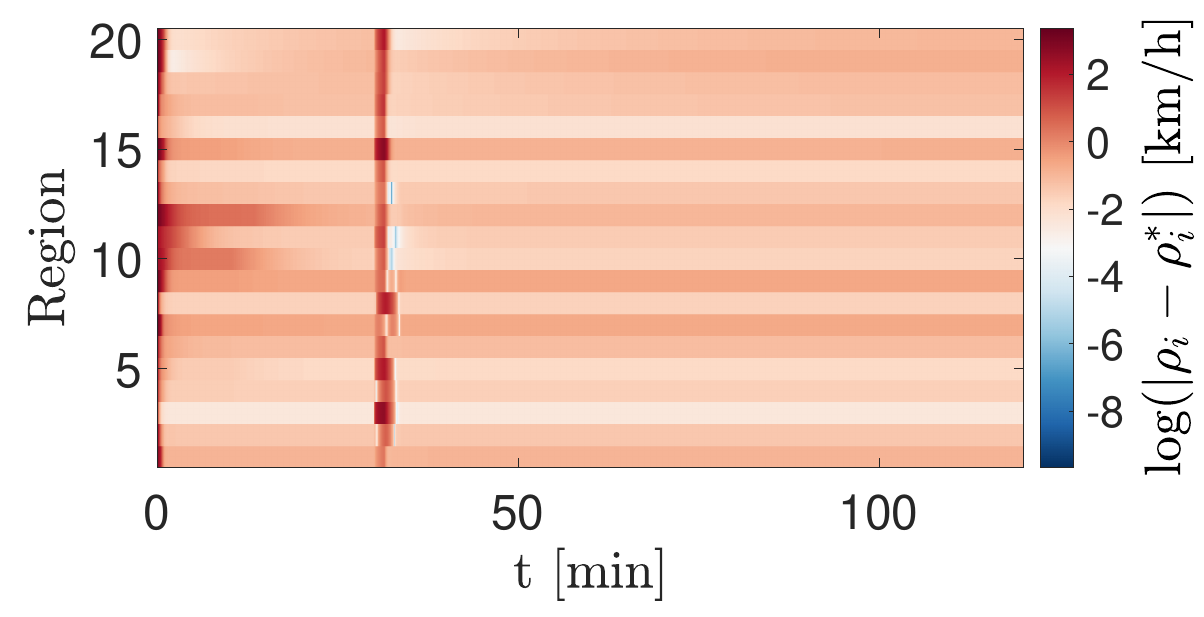}
\caption{Logarithmic heatmap density-error response from desired \Cref{tab:con_set_20} values, $[{\text{veh}}/{\text{km}}]$, showing rapid convergence to zero.}
\label{fig:e_r_20_heat}
\end{subfigure}\\
\vspace{-0.2mm}%
\begin{subfigure}{1\columnwidth}
\centering
\includegraphics[trim={0cm 8.5cm 1.2cm 9cm},clip,width=0.8\columnwidth]{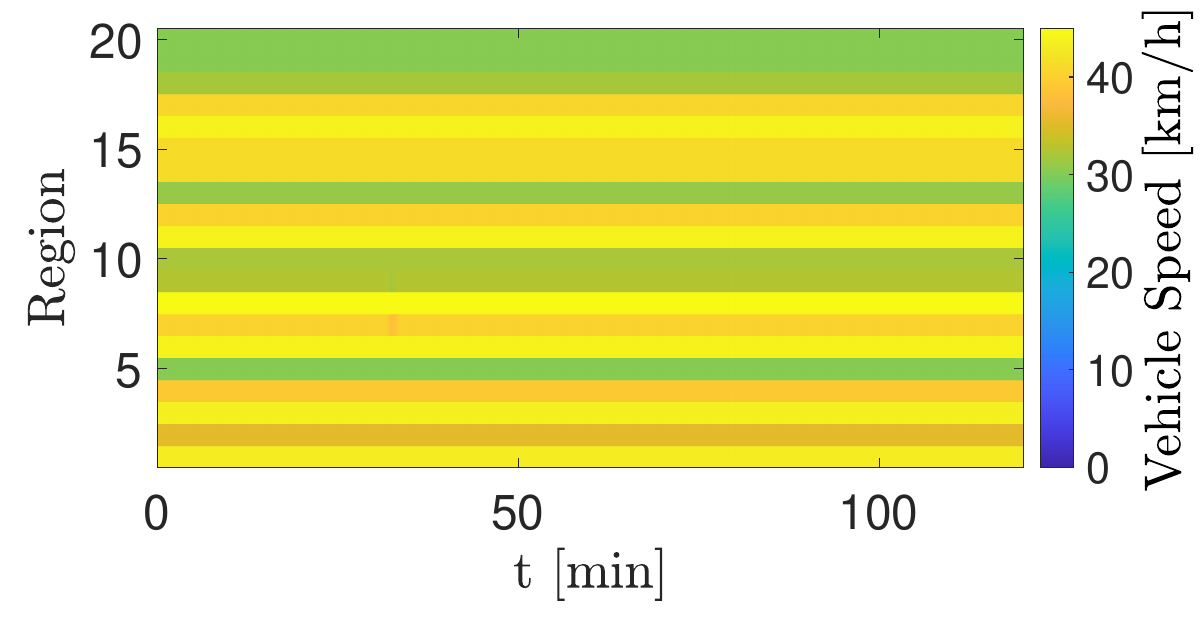}
\caption{Vehicle speed heatmap response, $[{\text{km}}/{\text{h}}]$, showing fast recovery to free-flow nominal operation.}
\label{fig:v_20_heat}
\end{subfigure}
\vspace{-2mm}
\caption{Heatmaps showing network log density-error, $[{\text{veh}}/{\text{km}}]$, and vehicle speed, $[{\text{km}}/{\text{h}}]$, evolution. 
Rapid convergence to the reference operating point is observed, with vehicle speeds in all regions stabilizing at nominal free-flow values.
A control disruption of 1.5~[min] duration at t = 30~[min] causes temporary congestion in regions 7, 9, 12 and 15 (see Fig.~\ref{fig:v_20_heat}), but the VAC scheme quickly restores the system to the reference set-points demonstrating rapid recovery.
During the disruption event unregulated vehicles enter the network from regions~1, 3 and 5.}
\label{fig:r_20}
\vspace{-3mm}
\end{figure}


\subsection{Assessing Network Stability Under Uncertainty\label{sec:uncert}}


Next, to demonstrate the robustness of the developed VAC framework while operating under uncertainty, 200 different  simulations are conducted on the 20-region network from \Cref{sec:scal} as follows.
The simulation scenario from \Cref{sec:num_sce_1} is repeated 200 times, each time with a different randomly generated uncertainty function, $d_i(\rho_i)$, abiding by the same conditions, as presented in \Cref{asm:d_prop}. 
The same fixed control gains are used across all simulations (shown in \Cref{tab:con_par_20_uncert}), allowing to demonstrate the framework's ability to handle a range of perturbation functions with constant gain parameters.
In analogy with \Cref{sec:scal}, the VAC schemes \eqref{eq:con_prop}, \eqref{eq:con_prop_non}, \eqref{eq:con_1st_ord}, \eqref{eq:con_2nd_ord}, and \eqref{sys_u} are employed for regions~1 to 4, 5 to 8, 9 to 12, 13 to 16 and 17 to 20, respectively.
The procedure for assessing network stability under uncertainty is outlined below.

Initially, the VAC scheme at each region is tuned only once, with gain values that satisfy condition \eqref{eq:eta} of \Cref{thm:1}.
For tuning, we use the uncertainty Lipschitz constant, $v^{d,L}_i\in\mathbb{R}_+$.
Specifically, for each region, the uncertainty perturbation function is assumed to be a polynomial with a maximum flow value of $\max\{d_i(\rho_i)\}=500\:[{\text{veh}}/{\text{h}}]$ (see Fig.~\ref{fig:g5} for an example visualization) occurring at a peak density equal to $0.2\rho_i^J$.
Using this function the worst-case uncertainty Lipschitz constant is calculated for each region and used for gain selection.
The resulting Lipschitz constant satisfied $v^{d,L}_i\in[54.1,76.7]$, see \Cref{tab:params_20}.
In the simulations, the peak of each uncertainty perturbation function employed, $d_i(\rho_i),i\in\mathcal{N}$ ($\max\{d_i(\rho_i)\}\leq500\:[{\text{veh}}/{\text{h}}],i\in\mathcal{N}$), corresponds to densities in $(0.2\rho_i^J,0.3\rho_i^J)$.
Hence, the randomly selected uncertainty functions are constrained to Lipschitz constants no larger than those used in the initial tuning. 
This tests whether the proposed VAC framework can stabilize the system under substantial uncertainty without further gain retuning.
\begin{figure}[t]
\centering
\includegraphics[trim={0.5cm 8.5cm 1.5cm 8cm},clip,width=0.9\columnwidth]{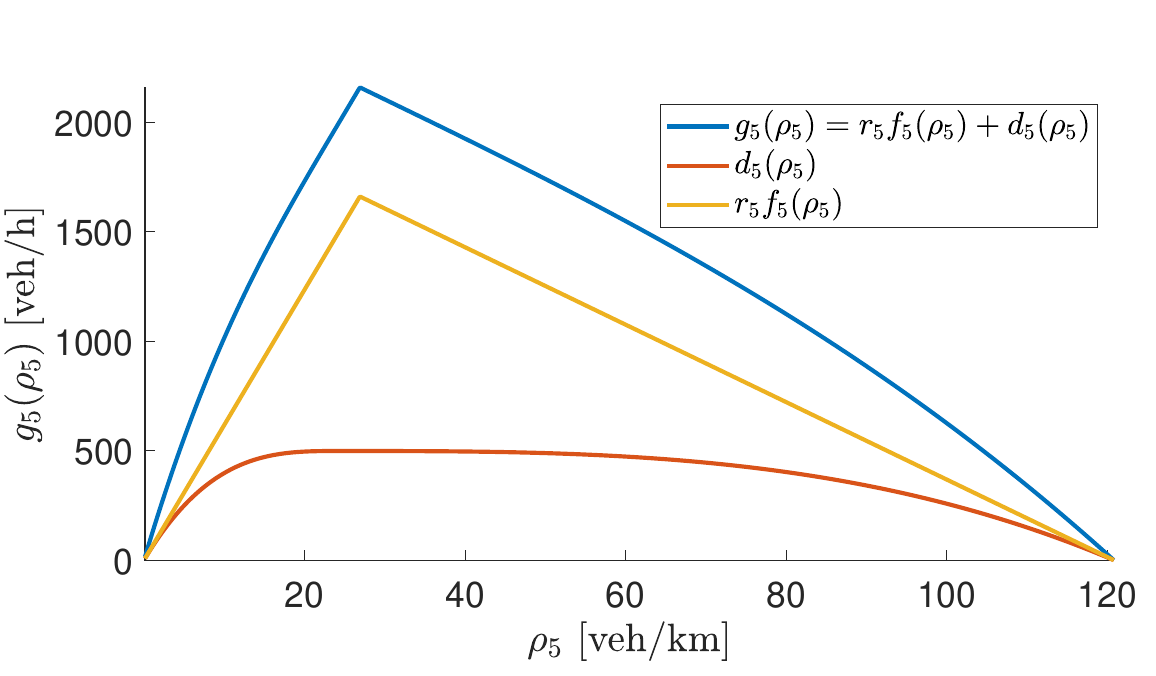}
\vspace{-2mm}
\caption{Worst case total inter-regional flow $[{\text{veh}}/{\text{h}}]$ for region 5, $g_5(\rho_5)$ (blue line). Estimated inter-regional flow $r_5 f_5(\rho_5)$ (orange line). Unknown uncertainty flow $d_5(\rho_5)$ (red line). 
Corresponding uncertainty Lipschitz constant $v_5^{d,L}=60.294$.
         }
         \label{fig:g5}
                  \vspace{-3mm}
\end{figure}

It is noted that in all simulations, the desired density set points correspond to free-flow conditions.
The density set-points and the parameters of the simulation are given in \Cref{tab:con_set_20_uncert,tab:params_20} respectively.
For presentation compactness, for each simulation and region, the absolute density deviation from steady-state, $e_{\rho_i}^{ss}(t)=|\rho_i(t)-\rho_i^{ss}|,i\in\mathcal{N}$, is calculated.
The maximum density deviation across all regions at each time step is then recorded as $e_{\rho}^{ss}(t)=\max_{i\in\mathcal{N}}\{e_{\rho_i}^{ss}(t)\}$.
Finally, for each time interval, the worst maximum density deviation across the 200 simulations, $e_{\rho,200}^{ss}(t)=\max_{1,..,200}\{e_{\rho}^{ss}(t)\}$, is saved into a new time series.
This condenses the results into a single plot capturing the largest error of all simulations under the VAC framework.

The series of experiments regarding robustness are presented next.
As depicted in Fig.~\ref{fig:r_200}, fast steady-state network stabilization is observed in the first 30 seconds.
At $t=30\,\text{[min]}$, the developed VAC framework becomes disengaged (for $1.5\,\text{[min]}$) and unregulated vehicles enter the network from regions 1, 3, and 5, at capacity flow rates, making some regions congested (regions 7, 9, 12 and 15, see Fig.~\ref{fig:v_20_heat}) and resulting in increased error.
When the VAC framework is re-engaged it reduces vehicle inflow to congested regions and increases vehicle inflow to less dense regions.
As a result, upon re-engagement the error quickly decreases in all the conducted simulations, as demonstrated in Fig.~\ref{fig:r_200}, validating the ability of the developed VAC design framework to stabilize the system, even at the presence of significant uncertainty and control disruptions.
\begin{figure}[t]
\centering
\includegraphics[trim={0.6cm 8.5cm 1.6cm 9cm},clip,width=0.8\columnwidth]{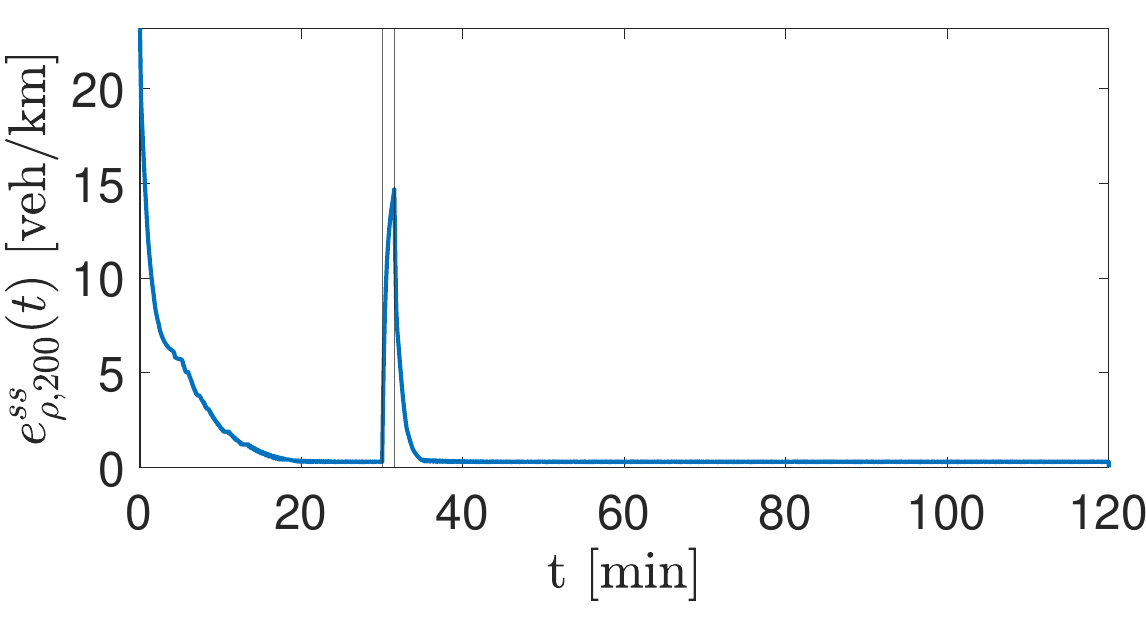}
\vspace{-2mm}
\caption{Worst maximum density deviation from steady-state across 200 simulations $[{\text{veh}}/{\text{km}}]$. 
Vertical lines indicate the $1.5\,\text{[min]}$ interval of disengagement of the VAC framework from all the regions. 
At this interval unregulated vehicles enter the network from regions~1, 3 and 5 at capacity flow rates.
         }
         \label{fig:r_200}
                  \vspace{-3mm}
\end{figure}

\subsection{Assessing VAC scheme sensitivity to parameter changes\label{sec:sens}}
To evaluate the ability of different VAC dynamic schemes to regulate demand in a stable manner, as well as their sensitivity to changes in VAC parameters, we conducted 50 simulations on the 20-region network from \Cref{sec:num_sce_1} using VAC dynamics with randomly generated parameters that satisfy the presented stability conditions. 
The desired operating point for this series of simulations is given in \Cref{tab:con_set_20}.
Similar to the previous sections, the VAC schemes \eqref{eq:con_prop}, \eqref{eq:con_prop_non}, \eqref{eq:con_1st_ord}, \eqref{eq:con_2nd_ord}, and \eqref{sys_u} are employed for regions~1 to 4, 5 to 8, 9 to 12, 13 to 16 and 17 to 20, respectively.
An indicative set of random gain values employed in one of the simulations, which satisfies the conditions of \Cref{thm:1}, is reported in \Cref{tab:con_par_20_sens}.
To condense the results into a single plot, the same method employed in \Cref{sec:uncert} is used to calculate the worst maximum density deviation from the desired operating point.
Pointedly, for each simulation the absolute density deviation from the desired operating point, $e_{\rho_i}^{*}(t)=|\rho_i(t)-\rho_i^{*}|,i\in\mathcal{N}$, is calculated for each region.
Using $e_{\rho_i}^{*}$, the maximum density deviation across all regions at each time step is then recorded as $e_{\rho}^{*}(t)=\max_{i\in\mathcal{N}}\{e_{\rho_i}^{*}(t)\}$.
Finally, for each time interval, the worst maximum density deviation across the 50 simulations, $e_{\rho,50}^{*}(t)=\max_{1,..,50}\{e_{\rho}^{*}(t)\}$, is saved into a new time series.

\begin{table}[t]
\centering
\caption{Random gain values (these gains satisfy \Cref{thm:1}) employed in one of the simulations in \Cref{sec:sens}}
\label{tab:con_par_20_sens}
\begin{tabularx}{\columnwidth}{cX}
\hline
VAC scheme: & Gains: \\
\hline 
\eqref{eq:con_prop} &$\boldsymbol{\eta}_{1,..,4}=[717.8,857.9,466.5,429.2]^T$\\
	& $\mathbf{c}_{1,..,4}=[15187,12019,2955,3669]^T$\\
\eqref{eq:con_prop_non} &$\boldsymbol{\eta}_{5,..,8}=[451.3,659.4,1040,550.9]^T$\\
	&$\mathbf{c}_{5,..,8}=[4240,10501,26829,6472]^T$\\
	&$\phi_i(\rho_i)=0.001\rho_i^3, i=5,..,8$\\
\eqref{eq:con_1st_ord} &$\boldsymbol{\eta}_{9,..,12}=[543.4,523,591.4,313.1]^T$\\
	&$\boldsymbol{\gamma}_{9,..,12}=[271.7,261.5,295.7,156.6]$\\
	&$\mathbf{c}_{9,..,12}=[20289,8820,9327,10032]^T$\\
	&$\tau_i=0.1, i=9,..,12$\\
\eqref{eq:con_2nd_ord} &$\boldsymbol{\eta}_{13,..,16}=[420.6,703,517.2,475.3]^T$\\
	&$\mathbf{c}_{13,..,16}=[6697,6949,16049,3113]^T$\\
	&$\tau_i=0.1, \kappa_i=0.2, i=13,..,16$\\
\eqref{sys_u} &$\boldsymbol{\rho}^{th,1}_{17,..,20}=[0.4,0.4,0.3,0.3]^T$\\
	&$\boldsymbol{\rho}^{th,2}_{17,..,20}=[131.4,126.5,118.5,100.8]^T$\\
	&$\mathbf{p}^{c,max}_{17,..,20}=[134.1,129.1,120.9,102.9]^T$\\
	&$\boldsymbol{\beta}^c_{17,..,20}=[535.7,564.3,535.2,550.6]^T$\\
	&$\mathbf{c}_{17,..,20}=[10228,11114,14287,14103]^T$\\
	&$T_{1,i}=0.01,T_{2,i}=0.1,T_{3,i}=0.2, \kappa_i=0.1,\gamma_i=1.023, i=17,..,20$\\
	& $\boldsymbol{\upsilon}_{17,..,20}=[0.001,0.001,0.001,0.001]^T$\\\hline
\end{tabularx}
\end{table}

The results are presented next.
At the initial stages of the simulations (up to $t=30\,\text{[min]}$), fast stabilization to the desired operating point is demonstrated in all cases, see Fig.~\ref{fig:r_50}.
The varying VAC parameters result in different transient responses for each simulation.
At $t=30\,\text{[min]}$ due to the control disruption event the VAC framework is disengaged from all the regions for $1.5\,\text{[min]}$, and unregulated vehicles enter the network from regions~1, 3 and 5 at capacity flow rates causing congestion to some regions.
Upon re-engagement of the VAC framework, fast recovery to the desired operating point is observed, as demonstrated in Fig.~\ref{fig:r_50}, validating the ability of the developed VAC design framework to stabilize the system, even at the presence of significant uncertainty and control disruptions.
These results indicate that the proposed VAC framework can effectively regulate vehicular flow, provided that the stability criterion is satisfied.

In conclusion, the simulation results clearly demonstrate the effectiveness of the proposed VAC strategies and the validity of the results presented in \Cref{sec:stab_cond} in the presence of modelling uncertainty, control disruptions and even constant user non-adherence to VAC commands.

\begin{figure}[t]
\centering
\includegraphics[trim={0.6cm 8.5cm 1.6cm 9cm},clip,width=0.8\columnwidth]{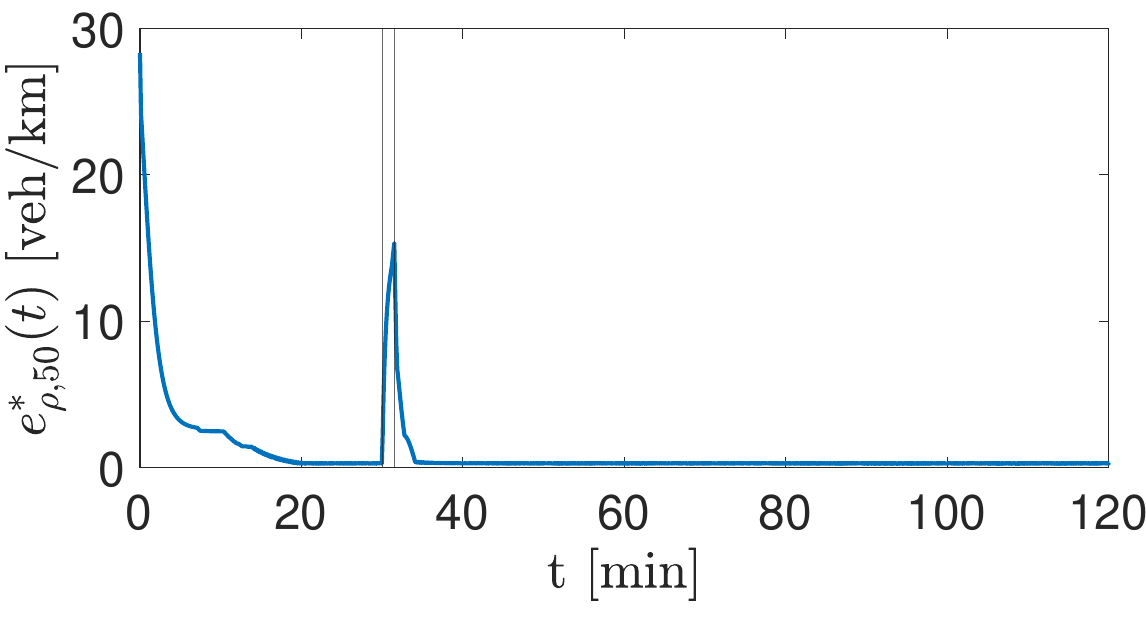}
\vspace{-2mm}
\caption{Worst maximum density deviation from desired operating point of \Cref{tab:con_set_20} across 50 simulations $[{\text{veh}}/{\text{km}}]$. 
Vertical lines indicate the $1.5\,\text{[min]}$ interval of disengagement of the VAC framework from all the regions. 
At this interval unregulated vehicles enter the network from regions~1, 3 and 5 at capacity flow rates.
         }
         \label{fig:r_50}
                  \vspace{-3mm}
\end{figure}

\section{Conclusions\label{sec:conclusions}}\color{black}
This work considered a broad class of decentralized vehicular admission control schemes (VAC) described by general nonlinear dynamics and examined the stability properties of their feedback connection with large-scale, heterogeneous regional, traffic networks described by uncertain, nonlinear, concave macroscopic fundamental diagrams.
Passivity theory was employed and scalable, locally verifiable conditions on the VAC dynamics were developed, that enable stability guarantees at the presence of modelling uncertainty.
Suitable examples of VAC dynamics that demonstrate the applicability of the proposed conditions were presented. 
These conditions provide a flexible framework for the design of VAC schemes, enabling heterogeneous VAC dynamics in each region, and thus complementing different demand operational objectives, and network infrastructure.
The developed analytic results were validated with numerical simulations on a 6 and a 20-region system demonstrating the effectiveness and applicability of the presented VAC framework.

A plethora of important directions for future work exist including, the incorporation of sensor noise into the model, and to validate the presented results using a dedicated traffic simulation software. 
Moreover, since the developed framework does not currently account for external demand, another interesting direction is to investigate dynamic demand scenarios and explore how VAC designs may adapt to accumulating demand at the boundary.
Furthermore, future research could incorporate state-dependent, dynamically varying outflow split functions to enable more realistic modelling of route choice dynamics.
Finally, future work will extend the proposed framework to incorporate dynamic uncertainty functions.




\appendices

\section{Proof of \Cref{lem:int}}\label{sec:Ap2}
To prove \Cref{lem:int} it suffices to show that for each $i\in\mathcal{N}$, the new VAC dynamics \eqref{eq:int}, are locally input strictly passive about their equilibrium values $(x^*_{\rho_i},-\rho^*_i,z_i^*)$.

\begin{proof}
To show local input strict passivity, the following storage function is employed
\begin{align}\label{eq:lyap_int}
V_i(x_{\rho_i},z_i)=\tilde{V}_i(x_{\rho_i})+\frac{\upsilon_i}{2}(z_i-z_i^*)^2
\end{align}
where $\tilde{V}_i(x_{\rho_i})$ is a storage function associated with \eqref{eq:int_1}-\eqref{eq:int_2}, according to \Cref{asm:con_dyn_int}, i.e. for a constant input $-\rho_i^*$ and steady state value $x^*_{\rho_i}$ of the \eqref{eq:int_1}-\eqref{eq:int_2} dynamics, $\tilde{V}_i(x_{\rho_i})$ satisfies \eqref{eq:stor_fun} and \eqref{eq:phi_bound}.
Hence,
\begin{subequations}
\begin{align}
V_i(x_{\rho_i},z_i)&>0\text{ in }\mathbb{R}\backslash \{x_{\rho_i}^*\}\times\mathbb{R}\backslash \{z_i^*\},\\
V_i(x_{\rho_i}^*,z_i^*)&=0.
\end{align}
\end{subequations}
Employing \eqref{eq:stor_fun_3}, \eqref{eq:phi_bound}, and \eqref{eq:int_3} the derivative of \eqref{eq:lyap_int} yields
\begin{subequations}
\begin{align}
\dot{V}_i=&\dot{\tilde{V}}_i+\upsilon_i(z_i-z_i^*)\dot{z}_i\\
\leq&(\tilde{u}_i-\tilde{u}^*_i)(-{\rho}_i-(-{\rho}^*_i))-{\eta}_i({\rho}_i-{\rho}^*_i)^2\nonumber\\
&+(z_i-z_i^*)(-{\rho}_i-(-\tilde{\rho}_i))\label{eq:lyap_int_1}
\end{align}
\end{subequations}
Since the density equilibrium $\rho_i^*$ satisfies \eqref{eq:int_eq_cond}, $\rho_i^*=\tilde{\rho}_i$, and  \eqref{eq:lyap_int_1} yields
\begin{align}\label{eq:lyap_int_2}
\dot{V}_i\leq&(\tilde{u}_i-\tilde{u}^*_i)(-{\rho}_i-(-{\rho}^*_i))-{\eta}_i({\rho}_i-{\rho}^*_i)^2\nonumber\\
&+(z_i-z_i^*)(-{\rho}_i-(-\tilde{\rho}_i))+({u}_i-{u}^*_i)(-{\rho}_i-(-{\rho}^*_i))\nonumber\\
&-({u}_i-{u}^*_i)(-{\rho}_i-(-{\rho}^*_i))
\end{align}
Employing \eqref{eq:int_4} and via $u_i^*=\tilde{u}_i^*+z_i^*$, \eqref{eq:lyap_int_2} yields
\begin{subequations}\label{eq:lyap_int_3}
\begin{align}
\dot{V}_i\leq&(\tilde{u}_i-\tilde{u}^*_i)(-{\rho}_i-(-{\rho}^*_i))-{\eta}_i({\rho}_i-{\rho}^*_i)^2\nonumber\\
&+(z_i-z_i^*)(-{\rho}_i-(-\tilde{\rho}_i))+({u}_i-{u}^*_i)(-{\rho}_i-(-{\rho}^*_i))\nonumber\\
&-(\tilde{u}_i+z_i-\tilde{u}_i^*-z_i^*)(-{\rho}_i-(-{\rho}^*_i))\\
\dot{V}_i\leq&({u}_i-{u}^*_i)(-{\rho}_i-(-{\rho}^*_i))-{\eta}_i({\rho}_i-{\rho}^*_i)^2
\end{align}
\end{subequations}
Hence, the new VAC dynamics formed by an integrator and \eqref{eq:int_1}-\eqref{eq:int_2} in the configuration given by \eqref{eq:int} are locally input strictly passive about their equilibrium values $(x^*_{\rho_i},-\rho_i,z_i^*)$.
\end{proof}

\section{Proof of \Cref{thm:1}\label{sec:Ap1}}

This appendix provides the proof of \Cref{thm:1}.

\begin{proof}
\Cref{asm:con_dyn} is employed to define a Lyapunov function for the feedback interconnection between \Cref{str:t_mod}, and the control dynamics \eqref{eq:con_dyn}.

Consider the following Lyapunov candidate $V:\mathbb{R}^{\sum_1^{\lvert\mathcal{N}\rvert} n_i}\times\mathbb{R}^{\lvert\mathcal{N}\rvert}\rightarrow\mathbb{R}_+$ for the overall feedback interconnection
\begin{align}\label{eq:L}
V(\boldsymbol{x}_{\boldsymbol{\rho}}(t),\boldsymbol{\rho}(t))&=\sum_{i\in\mathcal{N}}\big(\frac{L_i}{2}{(\rho_i(t)-{\rho^*_i})^2}+V_i(x_{\rho_i}(t))\big).
\end{align}
Onwards, time dependence is dropped for compactness.
Next we consider solutions to \Cref{str:t_mod}, \eqref{eq:con_dyn} within $\bar{\mathcal{M}}$;
an open subset that contains the equilibrium point, of the {connected} compact set $\mathcal{M}\subset\mathbb{R}^{\sum_1^{\lvert\mathcal{N}\rvert} n_i}\times\mathbb{R}^{\lvert\mathcal{N}\rvert}$, given by
\begin{align}\label{eq:loc_set}
\mathcal{M}=\{\boldsymbol{x}_{\boldsymbol{\rho}},\boldsymbol{\rho}&:V\leq\epsilon,x_{\rho_i}\in\mathcal{X}_i\cap\hat{\mathcal{X}}_i, -{\rho}_i\in\mathcal{U}_i, \forall i\in\mathcal{N}\}.
\end{align}
Note, that a suitable selection of $\epsilon$, can ensure that $\rho_i \in [0, \rho^J_i]$.
The reader is reminded that $\hat{\mathcal{X}}_i$ is the region of attraction of $x_{\rho_i}^*$,  $\mathcal{X}_i$ is an open neighbourhood of $x_{\rho_i}^*$, and  $\mathcal{U}_i$ is an open neighbourhood of $-\rho_i^*$.

Next we show that $\mathcal{M}$ is invariant.
Employing \Cref{asm:con_dyn} the following holds
\begin{subequations}
\begin{align}\label{eq:L_pd}
V(\boldsymbol{x}_{\boldsymbol{\rho}},\boldsymbol{\rho})&>0\;\text{in}\;\mathbb{R}^{\sum_1^{\lvert\mathcal{N}\rvert} n_i}\times\mathbb{R}^{\lvert\mathcal{N}\rvert}\backslash\{\boldsymbol{x}^*_{\boldsymbol{\rho}},\boldsymbol{\rho}^*\},\\
V(\boldsymbol{x}^*_{\boldsymbol{\rho}},\boldsymbol{\rho}^*)&=0.
\end{align}
\end{subequations}
Employing \eqref{eq:d_rho}, and invoking Assumption~\ref{asm:con_dyn} the derivative of \eqref{eq:L} is  given by
\begin{align}\label{eq:dL_1}
\dot{V}\leq&\sum_{i\in\mathcal{N}}(\rho_i-{\rho^*_i})\Big(-{r_i} f_i(\rho_i) -  d_i(\rho_i)\nonumber\\
&+\sum_{j\in\mathcal{P}_i}w_{ji} ({r_j} f_{j}(\rho_j)+d_j(\rho_j))+{u_i}\Big)\nonumber\\
&+(-{\rho}_i-(-{\rho}^*_i))(u_i-{u}^*_i)-{\theta}_i(-{\rho}_i-(-{\rho}^*_i)).
\end{align}
At equilibrium $\dot{\rho}_i=0$ and \eqref{eq:d_rho_eq_1} takes the following form
\begin{equation}\label{eq:rho_eq}
\sum_{j\in\mathcal{P}_i}w_{ji}( {r_j} f_{j}({\rho_j^*})+d_j(\rho_j^*))+{u_i^*}-{r_i} f_i({\rho_i^*}) -d_i(\rho_i^*)=0.
\end{equation}
Subtracting \eqref{eq:rho_eq} from \eqref{eq:dL_1} and employing \eqref{eq:phi_bound} yields
\begin{align}\label{eq:dL_2}
\dot{V}\leq&\sum_{i\in\mathcal{N}}( \rho_i-{\rho^*_i})\Big(-{r_i} \left(f_i(\rho_i)-f_i({\rho_i^*})\right)\nonumber\\
&-(d_i(\rho_i)-d_i(\rho_i^*))+\sum_{j\in\mathcal{P}_i}w_{ji} \big( {r_j}\left(f_{j}(\rho_j)-f_{j}({\rho_j^*})\right)\nonumber\\
&+\left(d_j(\rho_j)-d_j(\rho_j^*)\right)\big)+{u_i}-{u_i^*} \Big)\nonumber\\
&-( \rho_i-{\rho^*_i})(u_i-{u}^*_i)-{\eta}_i ( \rho_i-{\rho^*_i})^2.
\end{align}
Since both $f_{i}(\rho_i)$ and $d_{i}(\rho_i)$ satisfy the Lipschitz condition, with Lipschitz constants $v^L_i,v^{d,L}_i$ (see Assumptions~\ref{asm:d_prop}, \ref{asm:f_prop}), \eqref{eq:dL_2} yields
\begin{align}\label{eq:dL_3}
\dot{V}(t)\leq&\sum_{i\in\mathcal{N}}\Big(r_i v^L_i ( \rho_i-{\rho^*_i})^2+v_i^{d,L} ( \rho_i-{\rho^*_i})^2 \nonumber\\
&+\sum_{j\in\mathcal{P}_i}w_{ji}(r_j   v^L_j+v^{d,L}_j)\lvert  \rho_i-{\rho^*_i}\rvert \lvert  \rho_j-{\rho^*_j}\rvert \nonumber\\
&-{\eta}_i ( \rho_i-{\rho^*_i})^2\Big).
\end{align}
Employing \eqref{eq:a_ji} and the following perfect squares identity 
\begin{align*}
\beta xy=&{\beta x^2}/({2\xi})+{\xi\beta y^2}/{2} 
-\big(x\sqrt{{\beta}/({2\xi})}-y\sqrt{{\xi\beta}/{2}}\big)^2,
\end{align*}
with the constant $\xi\in\mathbb{R}_+$, yields
\begin{align}\label{eq:dL_4}
\dot{V}\leq&\sum_{i\in\mathcal{N}}\Big(-({\eta}_i-v_i^{d,L}-r_i v_i^{L}) ( \rho_i-{\rho^*_i})^2
\nonumber\\
&+\sum_{j\in\mathcal{P}_i}\big(\frac{a_{ji}}{2 \xi_{ji}}( \rho_i-{\rho^*_i})^2+\frac{\xi_{ji}a_{ji}}{2} ( \rho_j-{\rho^*_j})^2\nonumber\\
&-\left(( \rho_i-{\rho^*_i})\sqrt{\frac{a_{ji}}{2 \xi_{ji}}}-( \rho_j-{\rho^*_j})\sqrt{\frac{\xi_{ji}a_{ji}}{2}}\right)^2\big) \Big).
\end{align}
Additional manipulations result to 
\begin{align}\label{eq:dL_5}
\dot{V}\leq&-\sum_{i\in\mathcal{N}}\Big({\eta}_i-\big(v_i^{d,L}+r_i v_i^{L}
+\sum_{j\in\mathcal{P}_i}\frac{a_{ji}}{2 \xi_{ji}}\nonumber\\
&+\sum_{j\in\mathcal{S}_i}\frac{\xi_{ij}a_{ij}}{2}\big) \Big)( \rho_i-{\rho^*_i})^2.
\end{align}

If \eqref{eq:eta} holds then \eqref{eq:dL_5} is negative definite in $\mathcal{M}\setminus\{{\boldsymbol{x}_{\boldsymbol{\rho}}^{*},\boldsymbol{\rho}^*}\}$ and zero at $\boldsymbol{x}_{\boldsymbol{\rho}}=\boldsymbol{x}_{\boldsymbol{\rho}}^{*}$, $\boldsymbol{\rho}={\boldsymbol{\rho}^*}$, i.e. 
\begin{subequations}
\begin{align}
\dot{V}
&\leq-\sum_{i\in\mathcal{N}}\zeta_i ( \rho_i-{\rho^*_i})^2<0\;\text{in}\;\mathcal{M}\setminus\{\boldsymbol{x}_{\boldsymbol{\rho}}^{*},\boldsymbol{\rho}^*\},\label{eq:dL}\\
\zeta_i&={\eta}_i-\big(v_i^{d,L}
+r_i v_i^{L}+\sum_{j\in\mathcal{P}_i}\frac{a_{ji}}{2 \xi_{ji}}+\sum_{j\in\mathcal{S}_i}\frac{\xi_{ij}a_{ij}}{2}\big).
\end{align}
\end{subequations}
\Cref{eq:dL} and \Cref{asm:con_dyn} imply that for sufficiently small $\epsilon$, $V$ has a strict local minimum at $(\boldsymbol{x}_{\boldsymbol{\rho}}^{*},\boldsymbol{\rho}^*)$ and  moreover \eqref{eq:dL} implies that $\mathcal{M}$ is invariant, i.e., in this neighbourhood, $V$ is non-increasing on each state.
Hence, by Lasalle's Invariance Principle it is guaranteed that on $\mathcal{M}$, all solutions of the feedback interconnection between \Cref{str:t_mod}, and \eqref{eq:con_dyn}
that start on $(\boldsymbol{x}_{\boldsymbol{\rho}}^{}(0),\boldsymbol{\rho}(0))\in\bar{\mathcal{M}}$, converge to the largest invariant set, $\Lambda\subset\mathcal{M}$, where
\begin{equation}
\Lambda=\mathcal{M}\cap\big\{ (\boldsymbol{x}_{\boldsymbol{\rho}},\boldsymbol{\rho})\in\mathcal{M}:\dot{V}=0\big\}.
\end{equation}
At any point in $\Lambda$, $\dot{V}=0$, then by \eqref{eq:dL}, $\boldsymbol{\rho}=\boldsymbol{\rho}^*$ and this implies that $\dot{V}_i\leq0$ (see \eqref{eq:stor_fun_3}).
Since $\dot{V}_i \leq 0,\forall i\in\mathcal{N}$ and $\dot{V} = 0$, then each $\dot{V}_i = 0$.
Hence at equilibrium, $\boldsymbol{\rho}=\boldsymbol{\rho}^*$ is the corresponding constant input to the control dynamics \eqref{eq:con_dyn} that consecutively describe the input $\mathbf{u}$.
Moreover, by the definitions associated with \eqref{eq:con_dyn_eq}, when $\boldsymbol{\rho}=\boldsymbol{\rho}^*$ then $\boldsymbol{x}_{\boldsymbol{\rho}}^{}\to\boldsymbol{x}_{\boldsymbol{\rho}}^{*}$ since $\boldsymbol{x}_{\boldsymbol{\rho}}^{*}$ is locally asymptotically stable which by \Cref{asm:con_dyn} corresponds to a strict local minimum of $V_i$.
Hence, within $\Lambda$, $\boldsymbol{x}_{\boldsymbol{\rho}}^{}=\boldsymbol{x}_{\boldsymbol{\rho}}^{*}$ and by Lasalle's Invariance principle all solutions of the feedback interconnection  between \Cref{str:t_mod}, and the control dynamics \eqref{eq:con_dyn} with initial conditions $(\boldsymbol{x}_{\boldsymbol{\rho}}^{}(0),\boldsymbol{\rho}(0))\in\bar{\mathcal{M}}\subset\mathcal{M}$, converge to the largest invariant set of equilibrium points within ${\mathcal{M}}$.
\end{proof}

\section{Proof of \Cref{lemma_passivity}\label{sec:Ap3}}
This appendix provides the proof of \Cref{lemma_passivity}.

\begin{proof}
The condition $\gamma^c_{i} < \beta^c_i/K_i $ in \eqref{set_point_dynamics} 
implies that the $\mathcal{L}_2$-gain from $(\rho_i - \rho^*_i)$ to $(p^c_i(\rho_i) - p^{c}_i(\rho^*_i))$ is less than $\beta^c_i/K_i$ for any equilibrium $(\rho^*_i, p^{c}_i(\rho^*_i))$.
Hence, since \eqref{sys_G} has an $\mathcal{L}_2$-gain of $K_i$, it follows that the $\mathcal{L}_2$ gain from $(\rho_i - \rho^*_i)$ to $(u_{1,i} - u^*_{1,i})$ is less than $\beta^c_i$ for any feasible equilibrium pair $ (\rho^*_i, u^*_{1,i})$. This allows to show that the system from $-\rho_i$ to $u_i$ is input strictly passive about any equilibrium by directly applying \cite[Prop. 1]{kasis2016primary}.
\end{proof}



\bibliographystyle{IEEEtran}

\bibliography{sample} 

\end{document}